\documentclass[a4paper,10pt]{article}
\usepackage[utf8]{inputenc}

\usepackage{comment,xcomment,amssymb,amsmath,color}
\usepackage{rotating,xypic,array}
\usepackage{latexsym,mathtools,amsthm}
\usepackage{subfig}
\usepackage{tikz,fancyvrb}
\usepackage{booktabs}  
\usetikzlibrary{arrows.meta}
\usepackage{enumitem}
\usepackage{url} 
\usepackage{longtable}
\usepackage[section]{placeins}
\usepackage{graphicx}
 \usepackage{geometry}
\usepackage[longtable]{multirow}	

\title{Construction of nice nilpotent Lie groups}
\author{Diego Conti and Federico A. Rossi}

\newtheorem{theorem}{Theorem}[section]
\newtheorem{lemma}[theorem]{Lemma}

\newtheorem{corollary}[theorem]{Corollary}
\newtheorem{proposition}[theorem]{Proposition}
\theoremstyle{definition}
\newtheorem{definition}[theorem]{Definition}
\newtheorem{example}[theorem]{Example}
\theoremstyle{remark}
\newtheorem{remark}[theorem]{Remark}

\newcommand{\hash}{\#}

\newcommand{\abs}[1]{\left\vert#1\right\vert}
\newcommand{\R}{\mathbb{R}}
\newcommand{\im}{\mathrm{Im}\,}         
\newcommand{\lie}[1]{\mathfrak{#1}}     
\newcommand{\g}{\lie{g}}
\newcommand{\Z}{\mathbb{Z}}
\newcommand{\Q}{\mathbb{Q}}

\newcommand{\C}{\mathbb{C}}

\newcommand{\hook}{\lrcorner\,}

\newcommand{\Gtwo}{\mathrm{G}_2}

\newcommand{\GL}{\mathrm{GL}}

\newcommand{\gl}{\lie{gl}}

\newcommand{\Span}[1]{\operatorname{Span}\left\{#1\right\}}
\newcommand{\tran}[1]{\hspace{.2mm}\prescript{t\hspace{-.5mm}}{}{#1}}

\DeclareMathOperator{\Aut}{Aut}

\DeclareMathOperator{\diag}{diag}

\DeclareMathOperator{\coker}{coker}
\DeclareMathOperator{\logsign}{logsign}

\newcolumntype{C}{>{$}c<{$}}
\newcolumntype{L}{>{$}l<{$}}
\newcolumntype{R}{>{$}r<{$}}

\begin{document}
\VerbatimFootnotes
\maketitle

\begin{abstract}
We illustrate an algorithm to classify nice nilpotent Lie algebras of dimension $n$ up to a suitable notion of equivalence; applying the algorithm, we obtain complete listings for $n\leq9$. On every nilpotent Lie algebra of dimension $\leq 7$, we determine the number of inequivalent nice bases, which can be $0$, $1$, or $2$.

We show that any nilpotent Lie algebra of dimension $n$ has at most countably many inequivalent nice bases.
\end{abstract}

\renewcommand{\thefootnote}{\fnsymbol{footnote}}
\footnotetext{This work was partially supported by GNSAGA of INdAM and by PRIN $2015$ ``Real and Complex Manifolds: Geometry, Topology and Harmonic Analysis''.}
\renewcommand{\thefootnote}{\arabic{footnote}}

Nice nilpotent Lie algebras were introduced in \cite{LauretWill:Einstein} as a useful device in the construction of Einstein Riemannian solvmanifolds. Indeed, by \cite{Heber:noncompact,Lauret:Einstein_solvmanifolds} all such metrics are obtained by taking a standard extension of a nilpotent Lie group that carries a nilsoliton metric. The Ricci operator of any left-invariant metric for which a nice basis is orthogonal is diagonalized by that basis; this simplifies greatly the study of the nilsoliton condition (\cite{Nikolayevsky}). Conversely, nilsolitons with simple pre-Einstein derivation are nice \cite[Lemma 2.5]{Payne:Applications}. For similar reasons, nice nilpotent Lie algebras have been studied in the context of the Ricci flow \cite{LauretWill:diagonalization}. Recently, nice nilpotent Lie algebras have been used in the pseudoriemannian context to obtain explicit invariant Einstein metrics on nilpotent Lie groups (\cite{ContiRossi:EinsteinNice,ContiRossi:RicciFlat,ContiRossi:EinsteinNilpotent}).

Regardless of the nice condition, nilpotent Lie groups are an important source of examples in several areas of differential geometry, such as:
complex structures (\cite{Salamon:ComplexStructures,CavalcantiGualtieri:GeneralizedNilmanifolds,CorderoFernandezUgarte}),
special metrics (\cite{FinoPartonSalamon,RossiTomassini}),
deformations and cohomology (\cite{Rollenske:GeometryNilmanifold,BazzoniFernandezMunoz:NonFormalCoSymplectic,AngellaRossi:DComplexCohomology}),
metrics with curvature conditions (\cite{FernandezCulma}),
quaternionic geometry (\cite{DottiFino:AbelianHypercomplex,ContiMadsen:Harmonic}),
$\Gtwo$-structures (\cite{Fernandez:Calibrated,FinoLujan:TorsionFreeG22,ContiFernandez:calibrated}),
parahermitian structures (\cite{Hengesbach,ContiRossi:ricci}),
product structures (\cite{Andrada:ComplexProduct}),
geometric flows (\cite{BCFG:BoundaryHitchinHypoFlow,Fernandez-Culma2015:SolitonOn6DimNilmanifolds}).
What makes these applications possible is not only the general properties of nilpotent Lie groups such as the existence of compact quotients (\cite{Malcev}) whose minimal model is the Koszul complex of the Lie algebra (\cite{Nomizu:cohomology}), but also the fact that they are classified up to dimension $7$ (see \cite{Magnin, Gong}).
We are not aware of any systematic classification of higher-dimensional nilpotent Lie algebras at the moment of writing; we hope that the complete lists of nice nilpotent Lie algebras of dimension $8$ and $9$ presented in this paper  will be useful in the construction of  examples and the study of geometric problems.

We think of a nice nilpotent Lie algebra as a pair $(\g,\mathcal{B})$, where $\g$ is a real nilpotent Lie algebra and $\mathcal{B}=\{e_1,\dotsc, e_n\}$ a basis such that each bracket $[e_i,e_j]$ lies in the span of some $e_k$ depending on $i,j$ and each interior product $e_i\hook de^j$ lies in the span of some $e^h$, also depending on $i,j$. Two nice nilpotent Lie algebras are considered equivalent if there is a Lie algebra isomorphism that maps basis elements to multiples of basis elements.

Whilst nilpotent Lie algebras (over the reals) are classified up to dimension $7$, no similar classification exists for nice nilpotent Lie algebras, although the smaller class of nilpotent Lie algebras with simple pre-Einstein derivation and invertible Gram matrix is classified in \cite{KadiogluPayne:Computational}. It is known that up to dimension six all nilpotent Lie algebras except one admit at least one nice basis (see \cite{LauretWill:diagonalization,FernandezCulma}). In \cite{ContiRossi:EinsteinNilpotent}, we found $11$ examples of $7$-dimensional nilpotent Lie algebras that do not admit a nice basis.

In this paper we present an algorithm to carry out the classification of nice nilpotent Lie algebras of dimension $n$ (see Section~\ref{sec:algorithm}). We classify nice nilpotent Lie algebras of dimension $\leq 9$ up to equivalence (see the Appendix~\ref{App:Tabelle} and ancillary files); for $n>9$ the algorithm remains valid, but implementation meets practical limits and the resulting lists would presumably be too long for any practical use. By comparison with \cite{Gong}, we list the $7$-dimensional nilpotent Lie algebras that do not admit a nice basis and those that admit two inequivalent bases (Theorem~\ref{thm:notnice7}); even in dimension $6$, our algorithm provides a much quicker proof compared to the aforementioned classifications (Proposition~\ref{prop:two_inequivalent_bases}).

Our classification is based on the notion of \emph{nice diagram}, namely a type of labeled directed acyclic graph associated to each nice Lie algebra; diagrams associated to equivalent nice Lie algebras are naturally isomorphic. We present algorithms to classify nice diagrams and, for each nice diagram, the corresponding nice Lie algebras. The set of Lie algebras associated to a nice diagram can be empty, discrete or even contain continuous families; in the latter case, our algorithm guarantees that distinct families correspond to inequivalent nice Lie algebras (Theorem \ref{thm:inequivalent}).

In order to classify the nice Lie algebras associated to a given nice diagram $\Delta$ up to equivalence, we first consider the action of the group $D_n$ of  $n\times n$ diagonal matrices, where $n$ is the dimension of the Lie algebra. The set of nice Lie algebras associated to a diagram $\Delta$ is represented by an open set in a real vector space $V_\Delta$, parametrizing the nonzero structure constants relative to the nice basis. The natural action of $D_n$ can be used to normalize some structure constants to $\pm1$, producing a fundamental domain for this action, namely an open set in a finite union of affine spaces  (Proposition~\ref{prop:fundamental_domain}). The resulting elements of $V_\Delta$ define a nice nilpotent Lie algebra provided the Jacobi identity is satisfied; this is a system of quadratic equations in the structure constants. For $n\leq 8$ these equations reduce to linear equations thanks to the normalizations performed earlier; in dimension $9$ there are only $20$ diagrams for which quadratic equations survive the normalization; it turns out that only $12$ of them give rise to continuous families
of Lie algebras. The last step of our algorithm takes into account the action of the finite group of automorphisms of $\Delta$ to ensure that distinct families correspond to inequivalent nice Lie algebras.

A nilpotent Lie algebra can admit two or more inequivalent nice bases. We prove that, on a fixed nilpotent Lie algebra, the set of nice bases taken up to equivalence is at most countable (Corollary~\ref{cor:atmostcountable}). Moreover, comparing our classification to Gong's classification of nilpotent Lie algebras (see \cite{Gong}), we prove that this set has at most two elements in dimensions $n\leq7$ (Theorem~\ref{thm:notnice7}).

\smallskip
\textbf{Acknowledgments}
We thank Jorge Lauret and Tracy L. Payne for their useful suggestions and remarks.

\section{Nice Lie algebras and nice diagrams}
\label{sec:diagrams}
We work in the category $\mathcal{N}$ of nice nilpotent Lie algebras, whose objects are pairs $(\g,\mathcal{B})$, where $\g$ is a real nilpotent Lie algebra and $\mathcal{B}$ is a nice basis. A nice basis on a Lie algebra $\g$ is a basis  $\{e_1,\dotsc, e_n\}$ of $\g$ such that each $[e_i,e_j]$ is a multiple of a single basis element $e_k$ depending on $i,j$, and each $e_i\hook de^j$ is a multiple of a single $e^h$, depending on $i,j$; here, $\{e^1,\dotsc, e^n\}$ denotes the dual basis.

Clearly, replacing a basis element with a multiple does not affect this property. Thus, we define \emph{morphisms} of $\mathcal{N}$ as Lie algebra homomorphisms that map basis elements to multiples of basis elements. Invertible homomorphisms will be called \emph{equivalences} to avoid confusion with Lie algebra isomorphisms; note that two nice Lie algebras $(\g,\mathcal{B}), (\g',\mathcal{B}')$ may be isomorphic without being equivalent (see the remark below).  We can also think of $\mathcal{N}$ as the set of nilpotent Lie algebra structures on $\R^n$ such that the standard basis is nice, up to the group $\Sigma_n\ltimes D_n$, i.e. the semidirect product of the group of permutations in $n$ letters and the group of diagonal real matrices.

A Lie group will be said to be nice if its Lie algebra has a nice basis.
\begin{remark}\label{rem:nice6notUnique}
It is known that all nilpotent Lie algebra of dimension six except one have a nice basis (see \cite{LauretWill:diagonalization} or \cite{FernandezCulma}; see also Section~\ref{sec:algorithm}). The nice basis is not unique, even up to equivalence; indeed, consider the nice Lie algebra
\[\texttt{62:2}\qquad(0,0,0,0,e^{12},e^{34}).\] 
This notation means that the coframe dual to the nice basis  $e^1,\dotsc, e^6$ satisfies
\[de^1=0=\dots = de^4,\ de^5=e^{12}=e^1\wedge e^2,\ de^6=e^{34}=e^3\wedge e^4;\]
the string \texttt{62:2} is the name of the nice Lie algebra as explained in Section~\ref{sec:nice_classification}.

Under the change of coframe  $e^1-e^4, e^2+e^3, e^1+e^4, e^3-e^2, e^5+e^6,  e^5-e^6$, one can write the same Lie algebra as
\[\texttt{62:4a}\qquad(0,0,0,0,e^{13}+e^{24},e^{12}+e^{34}),\] 
which is clearly not equivalent.
\end{remark}
We aim at classifying nice nilpotent Lie algebras up to equivalence; our main tool will be a functor from $\mathcal{N}$ to a category of graphs.

Fix a nice nilpotent Lie algebra $\g$ (since we have defined nice Lie algebras as pairs, it is understood that a nice basis $\mathcal{B}$ is also fixed). Let $\g^i$ be the lower central series, $\g^0=\g$, $\g^{i+1}=[\g,\g^i]$. We recall from \cite{LauretWill:diagonalization} that $\g$ has \emph{type} $(a_1,\dotsc, a_s)$ if
\[(\dim \g, \dim \g^1,\dots, \dim \g^s)=(a_1+\dots + a_s, a_2+\dots + a_s, \dots, a_s).\]
Let $\mathcal{B}$ be the nice basis fixed on $\g$. We say that a subspace $V\subset\g$ is \emph{adapted} to $\mathcal{B}$ if it spanned by elements of the basis $\mathcal{B}$. If $V,W$ are subspaces adapted to $\mathcal{B}$, then $[V,W]$ is also  adapted to $\mathcal{B}$; this is because $[V,W]$ is spanned by brackets of basis elements, and each such bracket belongs to the basis, up to a scalar. This immediately implies:
\begin{lemma}
\label{lemma:adapted}
If $\g$  is a nice Lie algebra with nice basis $\mathcal{B}$, then each $\g^i$ is \emph{adapted} to $\mathcal{B}$.
\end{lemma}
Thus, we can always reorder a nice basis $\{e_1,\dotsc, e_n\}$ in such a way that
\[\g^1=\Span{e_i\mid i>a_1}, \quad \dots,\quad \g^s=\Span{e_i\mid i>a_1+\dots+a_{s-1}}.\]

To each nice nilpotent Lie algebra we can associate a directed graph $\Delta$, by the following rules:
\begin{itemize}
\item the nodes of $\Delta$ are the elements of the basis $\mathcal{B}$; in symbols, $N(\Delta)=\mathcal{B}$.
\item there is an arrow from $e_i$ to $e_j$ if $e_j$ is a nonzero multiple of some  $[e_i,e_h]$, i.e. $e_i\hook de^j\neq0$. In this case, we shall write $(e_i,e_j)\in E(\Delta)$, or simply $e_i\to e_j$.
\end{itemize}
By construction, $\Delta$ is a directed acyclic graph with no multiple arrows with the same source and destination; a graph with these properties will be called a \emph{diagram}. An \emph{isomorphism} of diagrams is an isomorphism of graphs, i.e. a pair of compatible bijections between nodes and arrows.

A \emph{labeled diagram} is a diagram $\Delta$ enhanced with a function from $E(\Delta)$ to $N(\Delta)$; the node associated to an arrow will be called its \emph{label}. Given a nice nilpotent Lie algebra, we can modify the above construction to give a labeled diagram by declaring that the arrow $e_i\to e_j$ has label $e_h$ when
$e_j$ is a multiple of some $[e_i,e_h]$; we will write $e_i\xrightarrow{e_h}e_j$.

An \emph{isomorphism of labeled diagrams} is an  isomorphism of diagrams $f$ such that whenever the arrow $A$ is labeled $e$, then the arrow $f(A)$ is labeled $f(e)$; two isomorphic labeled diagrams will be said to be \emph{equivalent}. We will denote by $\Aut(\Delta)$ the group of automorphisms of a labeled diagram; note that by construction $\Aut(\Delta)$ is a subgroup of $\Sigma_n$. It is clear that equivalent nice Lie algebras determine equivalent labeled diagrams.

The nodes of a diagram $\Delta$ have a natural filtration $N(\Delta)=N_0\supset N_1\supset \dots \supset N_s$, where $N_{i+1}$ contains all the nodes that are reached by at least one arrow in $N_i$. Like for nice Lie algebras, we will say a diagram has \emph{type} $(a_1,\dotsc, a_s)$ if
\[(\abs{N}, \abs{ N_1},\dots, \abs{N_s})=(a_1+\dots + a_s, a_2+\dots + a_s, \dots ,a_s).\]
For any $n$, diagrams of type $(n)$ contain no arrows; hence, they are equivalent.

Up to equivalence, it is no loss of generality to identify the nodes of a diagram with the numbers $\{1,\dotsc, n\}$. The labeled diagram of a nice Lie algebra clearly satisfies the following conditions:
\begin{enumerate}[label=(N\arabic*)]
\item\label{enum:condNice1} any two arrows with the same source have different labels;
\item\label{enum:condNice2} any two arrows with the same destination have different labels;
\item\label{enum:condNice3} if $i\xrightarrow{j}k$ is an arrow, then $i$ differs from $j$ and $j\xrightarrow{i}k$ is also an arrow.
\end{enumerate}
In order to state the fourth condition that the labeled diagram of a nice Lie algebra must satisfy, we need to introduce more language.

Given a diagram $\Delta$ satisfying~\ref{enum:condNice1}--\ref{enum:condNice3}, let $\mathcal{I}_\Delta$ be the set of the $I=\{\{i,j\},k\}$ such that $i\xrightarrow{j}k$; we shall write
\[E_I=e^{ij}\otimes e_k, \quad I=\{\{i,j\},k\},\ i<j.\]
Take the $D_n$-representation $V_\Delta$ freely generated by the $E_I$, $I\in\mathcal{I}_\Delta$.  To obtain an actual Lie algebra from a diagram, one needs to fix an element
\[c=\sum_{I\in \mathcal{I}_\Delta} c_IE_I,\]
that will determine the structure constants; whenever $I=\{\{i,j\},k\}$, we shall write
\[c_{ijk}=\begin{cases}
c_I,& i<j\\ 0 & i> j
          \end{cases}.
\]
We say a diagram has a \emph{double arrow} $k\xrightarrow{i,j} h$ if there is some $l$ for which $k\xrightarrow{l} h$ and $i \xrightarrow{j} l$ are arrows in the diagram with $k\neq i,j$; double arrows are parametrized by
\[\mathcal{I}_{\Delta\otimes\Delta}=\{(\{i,j\},k,h) \mid k\xrightarrow{i,j} h\}.\]
A double arrow $k\xrightarrow{i,j} h$ reflects the fact that $[[e_i,e_j],e_k]$ is a nonzero multiple of $e_h$. The exclusion of the case $k=i$ or $k=j$ is motivated by the fact that the Jacobi identity holds trivially for repeated indices. We shall write
\begin{equation}
\label{eqn:IstarJ}
I\star J=(\{i,j\},k,h), \quad I=\{\{i,j\},l\},\ J=\{\{l,k\},h\};
 \end{equation}
each element of $\mathcal{I}_{\Delta\otimes\Delta}$ can be written uniquely as $I\star J$, with $I,J\in\mathcal{I}_\Delta$.

We can define another acyclic directed graph $\Delta\otimes\Delta$ with the same nodes as $\Delta$ and whose arrows are the double arrows in $\Delta$; note that this graph can have multiple arrows, and its natural labeling takes values in the power set of $N(\Delta)$, rather than $N(\Delta)$. We define a representation $V_{\Delta\otimes\Delta}$ spanned by
\[e^{ij}\otimes e^k\otimes e_h, \quad (\{i,j\},k,h) \in \mathcal{I}_{\Delta\otimes\Delta}.\]
We have a natural quadratic, equivariant map $V_\Delta\to V_{\Delta\otimes \Delta}$,
\[c_IE_I\mapsto c_Ic_JE_{IJ},\]
where given $I,J$ as in \eqref{eqn:IstarJ},
\[E_{IJ}=
 \begin{cases}
  E_{I\star J} & l<k\\
  -E_{I\star J} & l>k.
 \end{cases}
\]
and $E_{IJ}=0$ when $I\star J$ is not defined.

In order to express the Jacobi identity, we will need to consider the alternating map
\[\mathfrak{a}\colon \Lambda^2(\R^n)^*\otimes (\R^n)^*\otimes \R^n\mapsto \Lambda^3(\R^n)^*\otimes \R^n.\]

\begin{proposition}
\label{prop:jacobi}
Given a labeled diagram $\Delta$ satisfying ~\ref{enum:condNice1}--\ref{enum:condNice3} and $c\in V_\Delta$, define \allowbreak $d\colon (\R^n)^*\to\Lambda^2(\R^n)^*$ by
\[de^k=\sum_{\{\{i,j\},k\}\in \mathcal{I}_\Delta}c_{ijk} e^{ij}, \quad i<j.\]
This defines a Lie algebra with diagram $\Delta$ if and only if all of the following hold:
\begin{itemize}
\item for each $I \in\mathcal{I}_\Delta$, $c_I\neq0$.
\item $\sum_{I,J\in\mathcal{I}_\Delta} c_Ic_JE_{IJ}$ lies in the kernel of $\mathfrak{a}$.
\end{itemize}
\end{proposition}
\begin{proof}
The Jacobi identity amounts to proving
\begin{multline*}
0=d^2e^h\otimes e_h=\sum_{\{\{l,k\},h\}\in \mathcal{I}_\Delta}c_{lkh} de^{lk}\otimes e_h
\\
 =\sum_{\{\{l,k\},h\},\{\{i,j\},l\}\in\mathcal{I}_\Delta}  c_{ijl}(c_{lkh}-c_{klh})e^{ijk} \otimes e_h
 =\sum_{I,J} c_Ic_J\mathfrak{a}(E_{IJ}). \qedhere
\end{multline*}
\end{proof}

Given a diagram $\Delta$ and $I\in\mathcal{I}_\Delta$, let $\alpha_I$ be the weight for the action of $D_n$ on $E_I$. Choose a total ordering on $\mathcal{I}_\Delta$ and let $M_\Delta$ be the matrix whose rows represent $\alpha_I$ in the basis (dual to) $e^1\otimes e_1,\dots, e^n\otimes e_n$. The matrix $M_\Delta$ is known as the \emph{root matrix} in the literature (up to a sign) and it encodes important properties of the associated Lie algebras (\cite{Nikolayevsky,Payne:ExistenceOfSolitonMetrics}); for instance, elements of its kernel correspond to derivations with eigenvectors $e_1,\dotsc, e_n$.

The root matrix also gives important information regarding the action of $D_n$ on $V_\Delta$ (which is essential for a classification up to equivalence, see Proposition~\ref{prop:EquivalentIFFautOrDn}). Indeed, let $m=\abs{\mathcal{I}_\Delta}$; then $M_\Delta$ is an $m\times n$ matrix which can  viewed as a Lie algebra homomorphism $d_n\to d_m$. This homomorphism realizes the correspondence between the natural action of $D_n$ on $V_\Delta\subset\Lambda^2T^*\otimes T$ and the action of $D_m$ via the diagram
\begin{equation*}
\xymatrix{
d_n\ar[d]_\exp \ar[r]^{M_\Delta} & d_m\ar[d]^\exp \\
D_n \ar[r]^{e^{M_\Delta}} & D_m
}
\end{equation*}
By construction the entries of $M_\Delta$ are $0$ or $\pm1$; by taking the obvious projection, we obtain a matrix $M_{\Delta,2}$ with entries in $\Z_2$.

\begin{example}
The diagram of Figure~\ref{figure:22487111} has double arrows
\[4\xrightarrow{1,2}7, \quad 2\xrightarrow{1,4}7,\quad 2\xrightarrow{1,3}6,\quad 1\xrightarrow{2,4}7,\quad 1\xrightarrow{2,3}6.\]
This means that the Jacobi identity contains an equation with three terms, i.e.
\[c_{123}c_{347}+c_{246}c_{617}+c_{415}c_{527}=0.\]
In this case, the equations admit a solution; up to equivalence, we find  the one-parameter family of Lie algebras
 \[\texttt{754321:9}\quad(0,0,(1-\lambda) e^{12},e^{13},\lambda e^{14}+e^{23},e^{24}+e^{15},e^{34}+e^{25}+e^{16}).\]
Note that the arrow $3\xrightarrow{1,3}7$ is missing because of our definition of double arrow.
\end{example}

 \begin{figure}[thp]
\caption{Example of a  nice diagram}
 \subfloat[The diagram $\Delta$]{
  \includegraphics[width=0.4\textwidth]{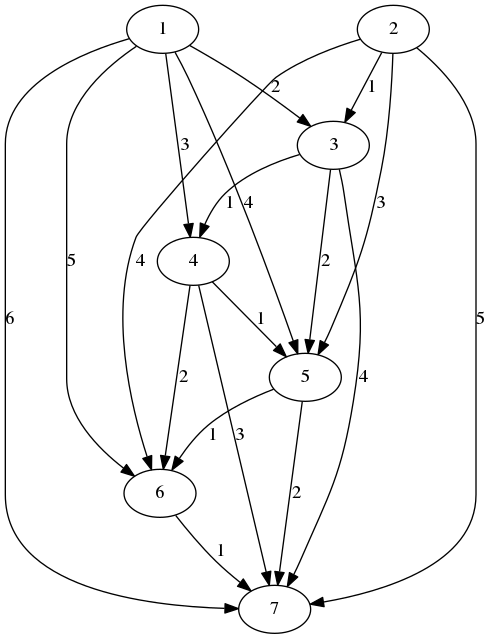}\label{figure:22487111}
 }\hfill
 \subfloat[The associated diagram $\Delta\otimes\Delta$]{
  \includegraphics[width=0.5\textwidth]{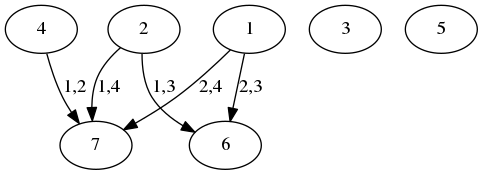}\label{figure:22487111_delta_delta}
 }
 \end{figure}

Notice that the kernel of $\mathfrak{a}$ is spanned by elements of the form
\[e^{ij}\otimes e^i \otimes e_h, \quad e^{ij}\otimes e^k\otimes e_h +  e^{kj}\otimes e^i\otimes e_h;\]
thus, terms of the form $e^{ij}\otimes e^k\otimes e_h$ with $k$ distinct from $i,j$ must appear in pairs or triples inside the sum $\sum_{I,J\in\mathcal{I}} c_Ic_JE_{IJ}$. The Jacobi identity implies a condition on the diagram:
\begin{enumerate}[resume*]
 \item\label{enum:condNice4} There do not exist four different nodes $i,j,k,v$ such that exactly one of
\[i\xrightarrow{j,k} v, \quad j\xrightarrow{k,i} v,\quad  k\xrightarrow{i,j} v\]
is a double arrow.
\end{enumerate}
In terms of the Gram matrix $U=M_\Delta \tran{M_\Delta}$ considered in \cite{Payne:ExistenceOfSolitonMetrics}, double arrows correspond to entries equal to $-1$; condition \ref{enum:condNice4} is equivalent to nonexistence of quadruples of multiplicity one in the sense of \cite{Payne:Methods}.

\begin{definition}
A labeled diagram will be called a \emph{nice diagram} if it satisfies conditions~\ref{enum:condNice1}--\ref{enum:condNice4}.
\end{definition}
Note that condition~\ref{enum:condNice4} is independent of~\ref{enum:condNice1}--\ref{enum:condNice3}, as shown by the example of Figure~\ref{fig:notverynice}. Summarizing, we have proved the following:
\begin{proposition}
\label{prop:nicenil}
For any nice nilpotent Lie algebra, the associated labeled diagram is nice.
\end{proposition}
\begin{remark}
\label{rmk:niceisnotnice}
The converse of Proposition~\ref{prop:nicenil} is not true. For an example of a nice diagram that does not correspond to any Lie algebra, see Figure~\ref{figure:no_lie_algebra}. In this case
\[V_\Delta=\{c_{123}e^{12}\otimes e_3 + c_{134}e^{13}\otimes e_4+c_{145}e^{14}\otimes e_5+c_{256}e^{25}\otimes e_6
+c_{346}e^{34}\otimes e_6
+c_{167}e^{16}\otimes e_7
+c_{357}e^{35}\otimes e_7\}.\]
Then
\begin{multline*}
\sum c_Ic_J E_{IJ}=
-c_{346}c_{167}e^{34}\otimes e^1\otimes e_7
-c_{256}c_{167} e^{25}\otimes e^1\otimes e_7
-c_{145}c_{357} e^{14}\otimes e^3\otimes e_7\\
+
c_{123}c_{357} e^{12}\otimes e^5\otimes e_7
-c_{145}c_{256} e^{14}\otimes e^2\otimes e_6
+c_{123}c_{346} e^{12}\otimes e^4\otimes e_6
\end{multline*}
The condition $\mathfrak{a}(c_Ic_J E_{IJ})=0$ is equivalent to the system
\begin{equation}
\label{eqn:Jacobifails}
\begin{gathered}
c_{256}c_{167}=c_{123}c_{357}\\
c_{123}c_{346}=-c_{145}c_{256}\\
c_{346}c_{167}=c_{145}c_{357}.
\end{gathered}
\end{equation}
This system has no solution with all the $c_I$ different from zero, as can be seen by multiplying the first two equations and dividing by the third.

Whilst any element of $V_\Delta$ that satisfies \eqref{eqn:Jacobifails} defines a nice Lie algebra, the associated diagram is $\Delta$ only when each $c_I$ is nonzero; therefore, there does not exist any nice Lie algebra with diagram $\Delta$ (see Proposition~\ref{prop:jacobi}).
\end{remark}

\begin{figure}[thp]
\caption{\label{fig:notverynice}Example of a labeled diagram that satisfies~\ref{enum:condNice1}--\ref{enum:condNice3} but not~\ref{enum:condNice4}}
\subfloat[The diagram $\Delta$]{
\includegraphics[width=0.3\textwidth]{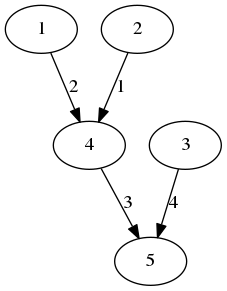}
}\hfill
\subfloat[The associated diagram $\Delta\otimes\Delta$]{
 \includegraphics[width=0.5\textwidth]{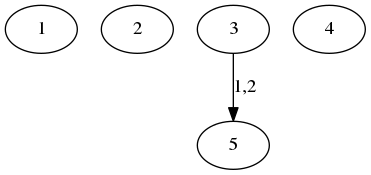}
}
\end{figure}

\begin{figure}[thp]
\caption{\label{figure:no_lie_algebra} A nice diagram that does not correspond to any Lie algebra}
\subfloat[The diagram $\Delta$]{
\includegraphics[width=0.3\textwidth]{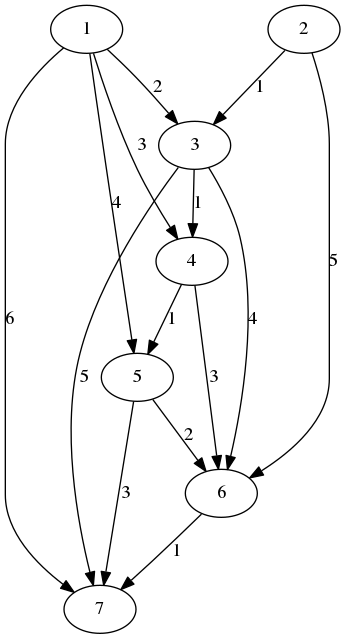}\label{figure:no_lie_algebra_delta}
}
\subfloat[The associated diagram $\Delta\otimes\Delta$]{
 \includegraphics[width=0.6\textwidth]{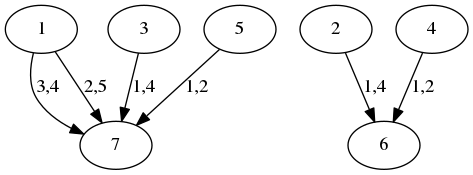}\label{figure:no_lie_algebra_delta_delta}
}
\end{figure}

\begin{remark}
\label{remark:surjective_no_jacobi}
Nice diagrams $\Delta$ such that $M_\Delta$ is surjective do not have any double arrows. Indeed, if $i\xrightarrow{j,k} v$ is a double arrow, condition~\ref{enum:condNice4} implies that some other double arrow $i\xrightarrow{h,m} v$ must exist. This establishes a relation of linear dependence between the rows of $M_\Delta$ corresponding to the brackets $[e_j,e_k]$, $[e_i,[e_j,e_k]]$, $[e_h,e_m]$ and $[e_i,[e_h,e_m]]$. A similar argument was used in \cite[Lemma 2.8]{KadiogluPayne:Computational} to prove a similar result in the special case that the Nikolaevsky derivation is simple.
\end{remark}

The group $\Aut(\Delta)$ acts linearly on $V_\Delta$ via
\[\sigma \cdot E_I = E_{\sigma^{-1}I};\]
this induces an action of $\Aut(\Delta)\ltimes D_n$. The following is obvious:
\begin{proposition}
\label{prop:EquivalentIFFautOrDn}
Two elements in $V_\Delta$ that satisfy the conditions of Proposition~\ref{prop:jacobi} define equivalent nice Lie algebras if and only if they are related by an element of $\Aut(\Delta)\ltimes D_n$.
\end{proposition}

\section{Classification algorithms}
\label{sec:algorithm}
In this section we illustrate algorithms to classify nice diagrams with $n$ nodes and nice Lie algebras of dimension $n$. We will assume nodes are numbered from $1$ to $n$, and fix a \emph{total} ordering on $\mathcal{P}(\{1,\dotsc, n\})$, for example through the canonical identification with numbers from $0$ to $2^{n}-1$; iterating through subsets is then a matter of iterating through integers.

The starting observation is the following:
\begin{proposition}
\label{prop:construct_diagrams}
Let $s>1$, $n=a_2+\dots + a_s$. Up to equivalence, any diagram of type $(a_1,\dotsc, a_s)$ can be obtained from a diagram of type $(a_2,\dots, a_s)$ by the following procedure:
\begin{itemize}
\item add $a_1$ nodes labeled $n,\dots, n+a_1$.
\item choose appropriate subsets $A_1\leq A_2\leq \dots \leq A_{a_1}$ of $\{1,\dots, n\}$ such that $A_1\cup \dots \cup A_{a_1}=\{1,\dotsc, n\}$.
\item for each $1\leq i\leq a_1$, add an arrow $(n+i)\to j$ whenever $j\in A_i$.
\end{itemize}
\end{proposition}
\begin{proof}
The condition $A_1\cup \dots \cup A_{a_1}=\{1,\dotsc, n\}$ is required for the type to be $(a_1,\dotsc, a_s)$.

Notice that interchanging $A_i$ with $A_j$ has the effect of interchanging the corresponding nodes $n+i$ and $n+j$; this explains why we can assume  $A_1\leq A_2\leq \dots \leq A_{a_1}$.
\end{proof}

\subsection{Enumerating nice diagrams}
For each $n$, it is easy to list the possible ways of writing $n$ as a sum of positive integers; for each such possibility $n=a_1+\dotsc + a_s$, we can  classify nice diagrams of type $(a_1,\dotsc, a_s)$ with the following algorithm.

\textbf{Step 1}. \emph{Classify diagrams of type $(a_1,\dotsc, a_s)$}.

There is a unique diagram of type $(a_s)$; working backwards, we apply  $s-1$ times the method of Proposition~\ref{prop:construct_diagrams}, iterating through all possible subsets at each step.

\textbf{Step 2}. \emph{Remove the diagrams where some nodes have an odd number of incoming arrows}.

This step is justified by the fact that any labeled diagram has an even number of incoming arrows at each node, because of \ref{enum:condNice3}.

\textbf{Step 3}. \emph{Eliminate isomorphic diagrams}.

Recall that two diagrams $\Delta,\Delta'$ are isomorphic if there are compatible bijections $N(\Delta)\to N(\Delta')$, $E(\Delta)\to E(\Delta')$; this implies that the resulting (families of) nice Lie algebras will be equivalent. The diagrams obtained from steps 1--2 may in general contain more representatives in a single isomorphism class. In order to eliminate them effectively (see subsection~\ref{subsection:implementation_notes}), we introduce an appropriate hash function, i.e. a map $\hash\colon N(\Delta)\to\Z$ that is invariant under automorphisms of $\Delta$, and define \[\hash(\Delta)=\sum_{e\in N(\Delta)} \hash(e).\]
By construction, $\Delta$ and $\Delta'$ can only be isomorphic if $\hash(\Delta)=\hash(\Delta')$, and a bijection $f\colon N(\Delta)\to N(\Delta')$ can only be an isomorphism if $f\circ\hash= \hash$.  Thus, for each pair of diagrams with the same hash code, we iterate through hash-preserving bijections and verify whether they induce diagram isomorphisms.

\textbf{Step 4}. \emph{For each diagram, compute the possible labelings}.

The idea is adding labels in pairs, iteratively, until the diagram is fully labeled. Formally, we consider \emph{partially labeled diagrams}, i.e. diagrams $\Delta$ with a function from the set of arrows $E(\Delta)$ to $V(\Delta)\cup \{\emptyset\}$, where the value $\emptyset$ represents ``no label''. A partial labeling that never takes the value $\emptyset$ will be called a \emph{complete labeling}.

To each partially labeled diagram $\Delta$, associate a set $C(\Delta)$ of completely labeled diagrams recursively as follows.

For each node $j$ denote by $V_j$ the set of nodes $v$ such $v\to j$ is unlabeled in $\Delta$. If all $V_j$ are empty, then $\Delta$ is completely labeled; set $C(\Delta)=\{\Delta\}$. Otherwise, consider the nodes for which $V_j$ is nonempty, let $j$ be the minimum among nodes that minimize $\abs{V_j}\neq0$ and add labels to two arrows ending at $j$ as follows. Let $i$ be the minimum of $V_j$. Let $W$ be the set of those $v\in V_j$ such that $v$ has no outgoing arrow with label $i$ and $i$ has no outgoing arrow with label $v$ in $\Delta$. If $W$ is empty, it is not possible to complete the labeling of $\Delta$, so set $C(\Delta)=\emptyset$; notice that choosing $j$ that minimizes $\abs{V_j}$ will generally make this condition occur earlier in the recursion. Otherwise, for each $w\in W$, let $\Delta_w$ be the  partially labeled diagrams obtained from $\Delta$ by adding the labels $w\xrightarrow{i} j$, $i\xrightarrow{w} j$; set recursively \[C(\Delta)=\bigcup_{w\in W} C(\Delta_w).\]
It is clear that the recursion has maximum depth $\abs{E(\Delta)}/2$, and that $C(\Delta)$ is the set of all the complete labelings of $\Delta$.

\textbf{Step 5}. \emph{Eliminate equivalent diagrams}.

This step is made necessary by the fact that Step 4 may produce different, equivalent labelings. We therefore proceed as in Step 3 to eliminate duplicates inside each $C(\Delta)$, with the only difference that the bijections $f$ that are considered are isomorphisms of labeled diagrams, i.e. they act compatibly on nodes, edges and labels.

\textbf{Step 6}. \emph{Eliminate diagrams for which~\ref{enum:condNice4} is violated}.

For each diagram $\Delta$ and each node $ v\in N(\Delta)$, we list all double arrows and apply the definition.

\subsection{Classifying nice Lie algebras}
\label{sec:classifying}
In order to determine the nice Lie algebras with associated diagram $\Delta$, we need to impose the conditions of Proposition~\ref{prop:jacobi}, i.e. the Jacobi identity. Since we are ultimately interested in classifying Lie algebras up to equivalence, it is convenient to factor out equivalence before imposing the Jacobi identity.

Consider the action of $\GL(n,\R)$ on  $\Lambda^2(\R^n)^*\otimes \R^n$; we will denote it by juxtaposition, so that $D_n$ as a subgroup of $\GL(n,\R)$ acts via $gc=e^{M_\Delta}(g)c$.  Isomorphism classes of Lie algebras with diagram $\Delta$ are elements of
\[(\GL(n,\R) \mathring{V}_\Delta )/\GL(n,\R),\]
where $\mathring{V}_\Delta$ is the complement of the union of the coordinate hyperplanes, i.e. the subset where each coordinate $c_I$ is nonzero.

Denoting by $\Z^*=\{\pm1\}$ the group of invertible integers, we can consider the restriction
\[e^{M_\Delta}\colon (\Z^*)^n\to(\Z^*)^{\mathcal{I}_\Delta}.\]
Denoting by $\logsign\colon \Z^*\to\Z_2$ the natural isomorphism, this restriction is identified with $M_{\Delta,2}\colon\Z_2^n\to\Z_2^{\mathcal{I}_\Delta}$.

It will be convenient to work with a fundamental domain $\mathring{W}\subset \mathring{V}_\Delta$, namely a submanifold that intersects each orbit of the action of $D_n$ in a single point. In fact, $\mathring{W}$ can be obtained by intersecting  $\mathring{V}_\Delta$ with a finite union of parallel affine spaces in $V_\Delta$. Recalling that rows of $M_\Delta$ are parametrized by $\mathcal{I}_\Delta$, we have:
\begin{proposition}
\label{prop:fundamental_domain}
Choose $\mathcal{J}_{\Delta,2}\subset\mathcal{J}_{\Delta}\subset \mathcal{I}_\Delta$ so that $\mathcal{J}_{\Delta,2}$ parametrizes a maximal set of $\Z_2$-linearly independent rows of $M_{\Delta,2}$ and $\mathcal{J}_{\Delta}$ parametrizes a maximal set of $\R$-linearly independent rows of $M_{\Delta}$. Set
\[
W=\left\{\sum c_I E_I\in V_\Delta\mid  c_I=1\ \forall I\in \mathcal{J}_{\Delta,2},\ c_I= \pm 1 \ \forall I\in \mathcal{J}_{\Delta}\setminus \mathcal{J}_{\Delta,2}\right\}.\]
Then $\mathring{W}=W\cap  \mathring{V}_\Delta$ is a fundamental domain in $\mathring{V}_\Delta$ for the action of $D_n$.
\end{proposition}
\begin{proof}
Composing with the obvious projection, one obtains a surjective homomorphism
\[(\Z^*)^n \xrightarrow{e^{M_\Delta}} (\Z^*)^{\mathcal{I}_\Delta} \to (\Z^*)^{\mathcal{J}_{\Delta,2}}.\]
It follows that up to the action of  $e^{M_{\Delta}}(D_n)$ any element $\sum c_IE_I$ of $\mathring{V} _\Delta$ can be assumed to satisfy $c_I>0$ whenever $I\in \mathcal{J}_{\Delta,2}$. Similarly, the composition
\[(\R^+)^n \xrightarrow{e^{M_\Delta}} (\R^+)^{\mathcal{I}_\Delta} \to (\R^+)^{\mathcal{J}_{\Delta}}\]
is surjective, and it follows that any coefficient $c_I$ can be normalized to $\pm1$  for $I\in \mathcal{J}_{\Delta}$.

It follows that $\mathring{W}$ as defined in the statement intersects every orbit. Now suppose the orbit of some $c$ in $\mathring{W}$ intersects $\mathring{W}$ in a point $c'=e^{M_\Delta}(\epsilon g) c$, where $\epsilon$ in $(\Z^*)^n$,  $g\in (\R^+)^n$. By definition of $\mathring{W}$,   the components $c_I$ and $c'_I$ coincide up to sign for each $I\in \mathcal{J}_\Delta$. Thus, each of the rows parametrized by $\mathcal{J}_\Delta$ annihilates $\log g$, and by maximality $M_\Delta(\log g)=0$. Similarly,  each of the rows parametrized by $\mathcal{J}_{\Delta,2}$ annihilates $\logsign \epsilon$; by maximality, this implies that $M_{\Delta,2}(\logsign \epsilon)$ is zero, i.e. $c'=e^{M_\Delta}(\epsilon)c=c$.
\end{proof}
Notice that it is no loss of generality to assume that $\mathcal{J}_{\Delta,2}\subset\mathcal{J}_\Delta$: rows that are linearly independent over $\Z_2$ are necessarily linearly independent over $\Q$, hence $\R$.

\begin{remark}
\label{remark:finiteunionofaffinespaces}
The characterization of $D_n$-orbits in $V_\Delta$ in terms of the root matrix was already given in \cite[Theorem 3.8]{Payne:Applications} and \cite[Corollary 3.5]{Payne:Methods}; the main improvement of Proposition~\ref{prop:fundamental_domain} is the explicit description of the fundamental domain $\mathring{W}$ as an open set in a finite union of affine spaces, which makes it possible to compute it with an algorithm.
\end{remark}

\begin{remark}
It is clear from Proposition~\ref{prop:fundamental_domain} (see also \cite[Theorem 3.8]{Payne:Applications}) that diagrams $\Delta$ of \emph{surjective type}, namely those for which $M_{\Delta,2}$ is surjective, are interesting from the point of view of Lie algebra classification, as in this situation the nice Lie algebra is determined uniquely by the diagram. These diagrams are also useful in the context of the construction of Einstein metrics (see \cite{ContiRossi:EinsteinNice}).

However the diagram can determine uniquely the nice Lie algebra even if $M_\Delta$ is not surjective (see Example~\ref{ex:631:6}).
\end{remark}

Potentially, each connected component in $W$ gives rise to a new family of Lie algebras with diagram $\Delta$. Recall that $\Aut(\Delta)\ltimes D_n$ acts naturally on $V_\Delta$; if $w\in V_\Delta$ defines a nice Lie algebra, then by Proposition~\ref{prop:EquivalentIFFautOrDn} it is equivalent to any element in its orbit $\{g\cdot w\mid g\in \Aut(\Delta)\ltimes D_n\}$.
The induced action of $\Aut(\Delta)$ on $\mathring{W}\cong \mathring{V}_\Delta/D_n$ will be denoted by juxtaposition.

Observing that the action of $\Aut(\Delta)$ on $\mathring{W}$ maps connected components to connected components, we can eliminate repeated families from our classification as follows:
\begin{theorem}
\label{thm:inequivalent}
Let $\Delta$ be a nice diagram and define $\mathring{W}$ as in Proposition~\ref{prop:fundamental_domain}. Let $\Aut(\Delta)$ act on the set of  connected components of $\mathring{W}$, and choose connected components $W_1,\dotsc, W_k$, one for each orbit. Let $B_j\subset W_j$ be the subset defined by the Jacobi equations. Then each element of $B_j$ defines a nice Lie algebra with diagram $\Delta$; up to equivalence, any nice Lie algebra with diagram $\Delta$ is obtained in this way. Moreover, if $j\neq k$, elements of $B_j$ and $B_k$ determine inequivalent nice Lie algebras.
\end{theorem}
\begin{proof}
Any nice Lie algebras with diagram $\Delta$ is determined by an element of $\mathring{V}_\Delta$ satisfying the Jacobi identity; two such elements determine equivalent nice Lie algebras when they are in the same orbit under the action of $\Aut(\Delta)\ltimes D_n$. Since $\mathring{W}$ is a fundamental domain for the action of $D_n$, two points of $\mathring{W}$ define equivalent nice Lie algebras if and only if they are in the same orbit under $\Aut(\Delta)$.
\end{proof}

\begin{remark}
Notice that there is no action of $\Aut(\Delta)$ on $W$, because diagram automorphisms do not generally preserve $\mathcal{J}_{\Delta,2}$. This is the reason to consider connected components of $\mathring{W}$, even though we are ultimately interested in removing redundant components from $W$.
\end{remark}

In order to compute the action of $\Aut(\Delta)$ on the set of connected components, we employ the following:
\begin{proposition}
\label{prop:ActionOfAutDelta}
Given $\sigma$ in $\Aut(\Delta)$ and $\epsilon$ in $(\Z^*)^{\mathcal{I}_\Delta\setminus \mathcal{J}_{\Delta,2}}$, set
\[w_\epsilon =\sum_{I\in \mathcal{J}_{\Delta,2}} E_I + \sum_{I\notin \mathcal{J}_{\Delta,2}} \epsilon_I E_I \in \mathring{V}_\Delta;\]
if $\delta$ is the unique element of $\im M_{\Delta,2}$ such that
\[(\delta +\logsign \sigma(w_\epsilon))_I=0, \quad I\in \mathcal{J}_{\Delta,2},\]
then
\[\sigma w_\epsilon =w_{\sigma\epsilon}, \quad (\sigma\epsilon)_I=(-1)^{\delta_I} (\sigma\cdot w_\epsilon)_I, \quad I\notin \mathcal{J}_{\Delta,2}.\]
\end{proposition}
\begin{proof}
Existence and uniqueness of $\delta$ follow from the definition of $\mathcal{J}_{\Delta,2}$. By construction, $\delta=M_{\Delta,2}(x_1,\dotsc, x_n)$ with
\[e^{M_\Delta}((-1)^{x_1},\dotsc, (-1)^{x_n}) w_\epsilon = w_{\sigma\epsilon}.\qedhere\]
\end{proof}

\begin{example}
\label{ex:631:6}
The Lie algebra \texttt{631:6} with diagram given in Figure~\ref{figure:2216} has    
\[
 M_\Delta =
\left(
\begin{array}{cccccc}
-1 & -1 & 0& 1 & 0 & 0 \\
-1 & 0 & -1 & 0& 1 & 0 \\
0 & -1 & 0 & 0 & -1 & 1 \\
0 & 0 & -1 & -1 & 0 & 1
\end{array}
\right).
\]
It is clear that the rank over $\R$ is three, and the first three rows are independent over both $\R$ and $\Z_2$. This gives the element
\[
 e^{12}\otimes e_4+e^{13}\otimes e_5 + e^{25}\otimes e_6+c_{346}e^{34}\otimes e_6\in V_\Delta.
\]
Then
\[c_Ic_J \mathfrak{a}(E_{IJ}) = -c_{346}e^{123}\otimes e_6 + e^{123}\otimes e_6;\]
this is only zero when $c_{346}=1$. Thus, up to equivalence the only nice Lie algebra with diagram $\Delta$ is
\[\texttt{631:6}\qquad(0,0,0,e^{12},e^{13},e^{34}+e^{25}).\] 
\end{example}
\begin{figure}[thp]
\caption{\label{figure:2216}The nice diagram \texttt{631:6}}}
\center{
\includegraphics[width=5cm]{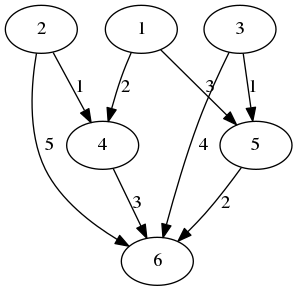}
\end{figure}

\begin{example}
\label{ex:SameDiagramInequivalentBases}
It is easy to see that the Lie algebras denoted by $N_{6,3,1a}$ and $N_{6,3,1}$ in \cite{Gong} admit inequivalent nice bases on the same diagram $\Delta$, namely
\begin{gather*}
\texttt{631:5a}\quad  (0,0,0,e^{12},e^{13},e^{24}+e^{35})\qquad N_{6,3,1a}\\
\texttt{631:5b}\quad  (0,0,0,- e^{12},e^{13},e^{35}+e^{24})\qquad  N_{6,3,1}.
\end{gather*}
This example shows that a nice diagram does not determine uniquely a nice Lie algebra.
\end{example}

\begin{example}
For an example where the action of $\Aut(\Delta)$ must be taken into account, consider the diagram \texttt{73:7} (see Table~\ref{TableNiceDim7}). In this case $M_\Delta$ has rank six and $W$ contains precisely four elements, each defining a Lie algebra, namely
\begin{align*}
   \texttt{i}&\quad(0,0,0,0,e^{23}+e^{14},e^{24}+e^{13},e^{12}+e^{34})\\
  \texttt{ii}&\quad(0,0,0,0,e^{14}+e^{23},-e^{13}+e^{24},e^{12}+e^{34})\\
 \texttt{iii}&\quad(0,0,0,0,-e^{14}+e^{23},e^{13}+e^{24},e^{12}+e^{34})\\
  \texttt{iv}&\quad(0,0,0,0,-e^{14}+e^{23},e^{24}-e^{13},e^{34}+e^{12})
\end{align*}
The automorphism $\sigma=(123)(567)$ cycles through  \texttt{i},\texttt{iii},\texttt{iv}, all isomorphic to $37D$, but \texttt{ii} is isomorphic to $37D_1$.
\end{example}

Summing up, we obtain the following algorithm:

\textbf{Step A}.
Choose a maximal set of $\Z_2$-linearly independent rows of $M_{\Delta,2}$, indexed by $\mathcal{J}_{\Delta,2}$, and choose $\mathcal{J}_{\Delta}\supset\mathcal{J}_{\Delta,2}$ parametrizing a maximal set of $\R$-linearly independent rows of $M_{\Delta}$. By Proposition~\ref{prop:fundamental_domain}, we are reduced to considering the set $\mathring{W}$ of elements $c_IE_I\in V_\Delta$ where
\begin{equation}
 \label{eqn:jacobi_ansatz}
c_I= \begin{cases}
        1 & I\in \mathcal{J}_{\Delta,2}\\
        \pm 1 & I\in \mathcal{J}_{\Delta}\setminus \mathcal{J}_{\Delta,2}\\
      \text{a nonzero constant} & \text{otherwise}.
       \end{cases}
  \end{equation}

\textbf{Step B}.
Determine the action of $\Aut(\Delta)$ on the set of connected components of $\mathring{W}$, and choose connected components $W_1,\dotsc, W_k$ of $\mathring{W}$, one for each orbit.

\textbf{Step C}.
On each component $W_j$, impose the Jacobi identity
\[\sum_{I,J} c_Ic_J\mathfrak{a}(E_{IJ})=0;\]
this is a system of polynomial equations in the $c_I$, some of which have degree less than two thanks to the assumptions~\eqref{eqn:jacobi_ansatz}. Neglecting quadratic equations for the moment, determine the subspace of $W_j$ where the linear equations are satisfied; this is a subset $L_j\subset W$ defined by linear equations and inequalities.

\textbf{Step D}.
For each nonempty $L_j$, consider the corresponding family of Lie algebras obtained by imposing the quadratic constraints originating from the Jacobi identity. Eliminate redundant families, namely those that only differ by changing the sign of the parameters. Each remaining family determines a row in the output table. By construction (see Theorem~\ref{thm:inequivalent}), different rows correspond to inequivalent nice Lie algebras.

\begin{remark}
Our algorithm can be adapted to give a classification over $\C$; in this case, all constants in $\mathcal{J}_\Delta$ can be normalized to $1$, so that the resulting space $W$ is connected. This means that each nice diagram gives rise at most to one family of nice Lie algebras. On the other hand, some nice diagrams that have been excluded in the real classification may give rise to complex solutions when the Jacobi identity admits solutions over $\C$ but not over $\R$.
\end{remark}

\subsection{Implementation notes}
\label{subsection:implementation_notes}
Our C++ implementation of the above algorithms is available at \url{https://github.com/diego-conti/DEMONbLAST}. We collect here some remarks on our implementation, listed according to the step to which they are related.

\emph{(Step 1)} Iteration through subsets is performed with an iterator-style class, represented internally by vectors of booleans.

Our algorithm to classify diagrams is based on iteration rather than recursion; this implies some repeated computations; for instance, in classifying diagrams with $5$ nodes, diagrams of type $(1,1)$ are computed twice, once for $(2,1,1,1)$ and one for $(3,1,1)$. This choice is partly motivated by the fact that the most computationally expensive step is the last one, where keeping track of already-computed diagrams does not give any advantage. On another note, the use of iteration greatly simplifies a parallelized implementation.

\emph{(Step 2)}. The operation of Step 2 is built into the last iteration in Step 1: at each node, the incoming arrows to be added are chosen in such a way that the overall number of incoming arrows at that node is even.

\emph{(Step 3)}. Elimination of isomorphic diagrams is best performed at an early stage, since it reduces the number of diagrams to be processed. For instance, for type $(2,2,4)$ one obtains $41$ diagrams, but only $9$ of them are pairwise nonisomorphic. The use of a hash function $\hash$, as opposed to brute-force iteration through the $n!$ bijections $N(\Delta)\to N(\Delta')$, leads to a considerable performance gain for an appropriate choice of $\hash$. In our implementation, we computed $\hash$ by counting concatenated arrows, i.e. sequences $v_1\to v_2\to \dots \to v_k$, so that $\hash(e)$ is defined in terms of the sequences
\[i_1(e),\dotsc, i_s(e), \quad o_1(e),\dotsc, o_s(e),\]
where $i_k(e)$ and $o_k(e)$ are the numbers of concatenated arrows of length $k$ respectively ending and beginning at $e$.

\emph{(Step C)}.
In order to determine whether each $L_j$ is empty, our program needs to determine whether a system of linear equalities and inequalities is consistent. Since $\coker M_\Delta$ is generally fairly small (e.g. $\dim \coker M_\Delta\leq 5$ for $n=8$, with the upper bound only attained by two nice diagrams), meaning that the dimension of $W$ is small, we adopted the simple strategy of solving the linear equations first, and then applying the Fourier-Motzkin method to the resulting system of inequalities, where the surviving number of unknowns is generally small.

\subsection{Comparison with \cite{KadiogluPayne:Computational}}
A similar algorithm was given in \cite{KadiogluPayne:Computational} to classify nilpotent Lie algebras such that the Nikolaevsky derivation is semisimple with  eigenvalues of multiplicity one and the Gram matrix is invertible; this condition implies that the eigenbasis is nice (\cite[Lemma 2.5]{Payne:Applications}) and the root matrix $M_\Delta$ is surjective (\cite[Lemma 2.8]{KadiogluPayne:Computational}) with at most $n-1$ rows. In this special situation the Jacobi identity follows automatically from the definition of nice diagram, which in this case rules out the existence of double arrows (see Remark~\ref{remark:surjective_no_jacobi}), or \emph{bad pairs} in the language of \cite{KadiogluPayne:Computational}.

The construction of \cite{KadiogluPayne:Computational} is inductive like ours; the absence of double arrows (bad pairs) is exploited to exclude some cases along the way. Moreover the hypotheses on the Nikolaevsky derivation imply that $\Aut(\Delta)$ is trivial, due to the invariance of said derivation under $\Aut(\Delta)$; therefore classification up to equivalence amounts to classification up to $D_n$. In addition, the same hypotheses are used to to show that the resulting Lie algebras are pairwise nonisomorphic.

Our construction is both more general and more efficient. Indeed, considering the symmetry between the nodes being added at each inductive step  enabled us to consider a smaller number of cases (reflected in the hypothesis $A_1\leq \dots \leq A_n$ of  Proposition~\ref{prop:construct_diagrams}). Moreover, our C++ implementation (written using sparse data structures from the STL) appears to be less memory-consuming than the matrix-based Matlab implementation illustrated in \cite{KadiogluPayne:Computational}.

In addition, we do not impose any restriction on either the rank of $M_\Delta$, the group $\Aut(\Delta)$ or the existence of double arrows. Our classification is made possible by the description of the fundamental domain $\mathring{W}$ (see Proposition~\ref{prop:fundamental_domain} and Remark~\ref{remark:finiteunionofaffinespaces}) and the characterization of the action of $\Aut(\Delta)$ in Proposition~\ref{prop:ActionOfAutDelta}.

\section{Classification of nice Lie algebras}
\label{sec:nice_classification}
In this section we  collect some applications of our algorithm and some theoretical remarks; in particular, we determine the number of inequivalent nice bases on each nilpotent Lie algebra of dimension up to $7$ and prove that nilpotent Lie algebras of dimension $n$ have at most countably many inequivalent nice bases. We identify nilpotent Lie algebras of dimension up to $7$ by  their name in Gong's classification (\cite{Gong}). For nice Lie algebras, we use a different name consisting of three parts: a sequence of integers, counting dimensions in the lower central series (LCS), a progressive number identifying the nice diagram, and possibly a letter to distinguish inequivalent families originating from the same diagram. Recall that in our language the choice of a nice Lie algebra implies the choice of a nice basis; therefore, a single entry in Gong's classification may correspond to more entries in ours.

The complete list of nice nilpotent Lie algebras up to dimension $7$ is given in the Appendix (see the ancillary files for dimensions $8$ and $9$).

As a first application, we improve Remark~\ref{rem:nice6notUnique} and complete the description of all $6$\--di\-men\-sio\-nal nice nilpotent Lie algebras. It is easy to check that $6$-dimensional nilpotent Lie algebras different from $N_{6,1,4}$ admit a nice basis; for example, in  Salamon's list (\cite{Salamon:ComplexStructures}), all Lie algebras are written in terms of a nice basis except
\begin{gather*}
 (0,0,12,13,0,14+23+25)\quad N_{6,1,4}\\
(0,0,0,12,13+ 14,24)\quad N_{6,2,9}
\end{gather*}
It is not difficult to see that $N_{6,2,9}$ admits a nice basis:
\[\texttt{632:3a}\qquad (0,0,0,e^{12},e^{14}+e^{23},e^{13}+e^{24}).\] 
Using our algorithm, we can easily verify the known fact that $N_{6,1,4}$ does not admit a nice basis (\cite[Proposition 2.1]{LauretWill:diagonalization}) and list the Lie algebras that admit more inequivalent nice bases  (see Table~\ref{table:6DimPiuNice}). We have then proved the following:
\begin{proposition}
\label{prop:two_inequivalent_bases}
Among nilpotent Lie algebras of dimension $6$:
\begin{itemize}
 \item  $N_{6,1,4}$ does not admit a nice basis;
 \item $N_{6,2,5}$, $N_{6,3,1}$ and $N_{3,2}\oplus N_{3,2}$ admit exactly two inequivalent bases;
 \item the remaining Lie algebras admit exactly one nice basis up to equivalence.
 \end{itemize}
\end{proposition}

\begin{table}[thp]
\centering
\caption{6-dimensional nilpotent Lie algebras with inequivalent bases\label{table:6DimPiuNice}}
\begin{tabular}{>{\ttfamily}l L L}
\toprule
\textnormal{Name} & \g & \text{Gong \cite{Gong}} \\
\midrule
\multicolumn{3}{c}{Type: $2121$}\\
6431:2b&0,0,e^{12},- e^{13},e^{23},e^{14}+e^{25}& N_{6,2,5}\\
6431:3&0,0,e^{12},e^{13},e^{23},e^{24}+e^{15}& N_{6,2,5}\\[2pt]
\multicolumn{3}{c}{Type: $321$}\\
631:5b&0,0,0,- e^{12},e^{13},e^{35}+e^{24}& N_{6,3,1}\\
631:6&0,0,0,e^{12},e^{13},e^{25}+e^{34}& N_{6,3,1}\\[2pt]
\multicolumn{3}{c}{Type: $42$}\\
62:2&0,0,0,0,e^{12},e^{34}& N_{3,2}\oplus N_{3,2}\\
62:4a&0,0,0,0,e^{13}+e^{24},e^{12}+e^{34}& N_{3,2}\oplus N_{3,2}\\
\bottomrule
\end{tabular}
\end{table}

Comparing the list of nice nilpotent Lie algebras (see Table~\ref{TableNiceDim7}) with the classification of $7$-dimensional nilpotent Lie algebras in \cite{Gong}, we obtain the following:
\begin{theorem}
\label{thm:notnice7}
The nilpotent Lie algebras of dimension $7$ listed in Table~{\ref{tabella7NOnice}} do not admit a nice basis; those listed in Table~\ref{table:tabella7MoreNice} admit exactly two inequivalent nice bases; the remaining nilpotent Lie algebras of dimension 7 admit exactly one nice basis up to equivalence.
\end{theorem}
\begin{proof}
Table~{\ref{tabella7NOnice}} contains nilpotent Lie algebras that do not appear in Table~\ref{TableNiceDim7}; Table~\ref{table:tabella7MoreNice} contains those that appear on two distinct rows of Table~\ref{TableNiceDim7}. It remains to show that nilpotent Lie algebras that appear on exactly one row in Table~\ref{table:tabella7MoreNice} do not admit inequivalent nice bases. This will be verified in Proposition~\ref{prop:finite7}.
\end{proof}

\FloatBarrier
\begin{footnotesize}
{\setlength{\tabcolsep}{2pt}
\begin{longtable}[htpc]{L L L}
\caption{7-dimensional nilpotent Lie algebras without a nice basis\label{tabella7NOnice}}\\
\toprule
\text{LCS Dim.} & \g & \text{Gong \cite{Gong}} \\
\midrule
\endfirsthead
\multicolumn{3}{c}{\tablename\ \thetable\ -- \textit{Continued from previous page}} \\
\toprule
\text{LCS Dim.} & \g & \text{Gong \cite{Gong}} \\
\midrule
\endhead
\bottomrule\\[-7pt]
\multicolumn{3}{c}{\tablename\ \thetable\ -- \textit{Continued to next page}} \\
\endfoot
\bottomrule\\[-7pt]
\endlastfoot
754321&0,0,e^{12},e^{13},e^{14},e^{15}+e^{23},e^{16}+e^{23}+e^{24}&123457E\\
754321&0,0,e^{12},e^{13},e^{14},e^{15}+e^{23},e^{16}+e^{24}+e^{25}-e{34}&123457F\\
754321&0,0,e^{12},e^{13},e^{14}+e^{23},e^{15}+e^{24},e^{16}+e^{23}+e^{25}&123457H\\
754321&0,0,e^{12},e^{13},e^{14}+e^{23},e^{15}+e^{24},-e^{16}+e^{23}-e^{25}&123457H_1\\

75421&0,0,e^{12},e^{13},e^{14}+e^{23},e^{15}+e^{24},e^{23}&23457E\\
75421&0,0,e^{12},e^{13},e^{14}+e^{23},e^{25}-e^{34},e^{23}&23457F\\

75421&0,0,e^{12},e^{13},e^{14},e^{23},e^{16}+e^{25}+e^{24}-e^{34}&13457G\\
75421&0,0,e^{12},e^{13},e^{14},e^{23},e^{15}+e^{25}+e^{26}-e^{34}&13457I\\

75421&0,0,e^{12},e^{13},e^{23},e^{15}+e^{24},e^{14}+e^{16}+e^{25}+e^{34}&12457J\\
75421&0,0,e^{12},e^{13},e^{23},e^{15}+e^{24},e^{14}+e^{16}-e^{25}+e^{34}&12457J_1\\
75421&0,0,e^{12},e^{13},e^{23},e^{15}+e^{24},e^{14}+e^{16}+e^{34}&12457K\\
75421&0,0,e^{12},e^{13},e^{23},e^{15}+e^{24},e^{14}+e^{16}+\lambda e^{25}+e^{26}+e^{34}-e^{35}&12457N\\
75421&0,0,e^{12},e^{13},e^{23},-e^{14}-e^{25},e^{16}+e^{25}-e^{35}&12457N_1\\
75421&0,0,e^{12},e^{13},e^{23},-e^{14}-e^{25},e^{15}+e^{16}+e^{24}+\lambda e^{25}-e^{35}&12457N_2\ \lambda\geq0\\

74321&0,0,e^{12},e^{13},e^{14},0,e^{15}+e^{23}+e^{26}&13457B\\
74321&0,0,e^{12},e^{13},e^{14}+e^{23},0,e^{15}+e^{24}+e^{26}&13457D\\

74321&0,0,e^{12},e^{13},0,e^{14}+e^{25},e^{35}+e^{16}&12457B\\
74321&0,0,e^{12},e^{13},0,e^{14}+e^{23}+e^{25},e^{16}+e^{24}+e^{35}&12457E\\
74321&0,0,e^{12},e^{13},0,e^{14}+e^{23}+e^{25},e^{26}-e^{34}&12457F\\
74321&0,0,e^{12},e^{13},0,e^{14}+e^{23}+e^{25},e^{15}+e^{26}-e^{34}&12457G\\

74321&0,0,0,e^{12},e^{14}+e^{23},e^{15}-e^{34},e^{16}+e^{23}-e^{35}&12357B\\
74321&0,0,0,e^{12},e^{14}+e^{23},e^{15}-e^{34},e^{16}-e^{23}-e^{35}&12357B_1\\

7431&0,0,e^{12},e^{13},0,e^{25}+e^{23},e^{14}&2457E\\
7431&0,0,e^{12},e^{13},0,e^{14}+e^{23},e^{23}+e^{25}&2457J\\
7431&0,0,e^{12},0,e^{13},e^{23}+e^{24},e^{15}+e^{16}+e^{25}+\lambda e^{26}+e^{34}&1357S\ \lambda\leq0\\

7321&0,0,e^{12},e^{13},0,e^{14}+e^{23}+e^{25},0&2457 N_{6,1,4}\oplus\R\\
7421&0,0,e^{12},e^{13},0,e^{14}+e^{25}+e^{23},e^{15}&2457D\\
7421&0,0,0,e^{12},e^{14}+e^{23},e^{23},e^{15}-e^{34}&2357A\\

7421&0,0,e^{12},0,e^{23},e^{14},e^{16}+e^{26}+e^{25}-e^{34}&1357H\\
7421&0,0,e^{12},0,e^{24}+e^{13},e^{14},e^{15}+e^{23}+\frac{1}{2}e^{26}+\frac{1}{2}e^{34}&1357L\\
7421&0,0,e^{12},0,e^{24}+e^{13},e^{14},e^{15}+\lambda e^{23}+e^{34}+e^{46}&1357N\\

742&0,0,0,e^{12},e^{13},e^{14}+e^{24}-e^{35},e^{25}+e^{34}&247H_1\\
742&0,0,0,e^{12},e^{13},e^{15}+e^{35},e^{25}+e^{34}&247J\\

741&0,0,0,e^{12},e^{23},-e^{13},e^{15}+e^{16}+e^{26}-2e^{34}&147D\\
741&0,0,0,e^{12},e^{23},-e^{13},-\lambda e^{16}+\lambda e^{25}+2e^{26}-2e^{34}&147E_1\,\lambda\!>\!1,\lambda\!\neq\!2\\

7321&0,0,0,e^{12},e^{14}+e^{23},0,e^{15}-e^{34}+e^{36}&1357B\\%
7321&0,0,0,e^{12},e^{14}+e^{23},0,e^{15}-e^{34}+e^{24}+e^{36}&1357C\\%

732&0,0,e^{12},0,0,e^{13}+e^{14},e^{15}+e^{23}&257I\\
732&0,0,e^{12},0,0,e^{13}+e^{14}+e^{25},e^{23}+e^{15}&257J_1\\
\end{longtable}
}
\end{footnotesize}

\FloatBarrier
\begin{footnotesize}
{\setlength{\tabcolsep}{2pt}
\begin{longtable}[htpc]{>{\ttfamily}l L L}
\caption{7-dimensional nilpotent Lie algebras admitting two nice bases\label{table:tabella7MoreNice}}\\
\toprule
\textnormal{Name} & \g & \text{Gong \cite{Gong}} \\
\midrule
\endfirsthead
\multicolumn{3}{c}{\tablename\ \thetable\ -- \textit{Continued from previous page}} \\
\toprule
\text{Name} & \g & \text{Gong \cite{Gong}} \\
\midrule
\endhead
\bottomrule\\[-7pt]
\multicolumn{3}{c}{\tablename\ \thetable\ -- \textit{Continued to next page}} \\
\endfoot
\bottomrule\\[-7pt]
\endlastfoot
\multicolumn{3}{c}{Type: $2122$}\\*
7542:1&0,0,e^{12},e^{13},e^{23},e^{14},e^{25}&2457L\\
7542:3a&0,0,e^{12},e^{13},e^{23},e^{24}+e^{15},e^{14}+e^{25}&2457L\\
\multicolumn{3}{c}{Type: $3121$}\\*
7431:7b&0,0,0,e^{12},- e^{14},e^{24},e^{26}+e^{15}&N_{6,2,5}\oplus\R\\
7431:8&0,0,0,e^{12},e^{14},e^{24},e^{25}+e^{16}&N_{6,2,5}\oplus\R\\
7431:10b&0,0,0,e^{12},-e^{14}+e^{23},e^{13}+e^{24},e^{26}+e^{15}&1357QRS_1\ \lambda=1\\
7431:11b&0,0,0,e^{12},e^{24}-e^{13},e^{23}+e^{14},e^{15}+e^{26}&1357QRS_1\ \lambda=1\\
\multicolumn{3}{c}{Type: $322$}\\*
742:5&0,0,0,e^{12},e^{13},e^{24},e^{35}&247F\\
742:18a&0,0,0,e^{12},e^{13},e^{25}+e^{34},e^{24}+e^{35}&247F\\
\multicolumn{3}{c}{Type: $331$}\\*
741:3b&0,0,0,- e^{12},e^{13},e^{23},e^{24}+e^{35}&247P\\
741:4&0,0,0,e^{12},e^{13},e^{23},e^{25}+e^{34}&247P\\
\multicolumn{3}{c}{Type: $421$}\\*
731:14b&0,0,0,0,- e^{12},e^{13},e^{36}+e^{25}&N_{6,3,1}\oplus\R\\
731:15&0,0,0,0,e^{12},e^{13},e^{26}+e^{35}&N_{6,3,1}\oplus\R\\
731:16b&0,0,0,0,- e^{12},e^{13},e^{36}+e^{25}+e^{14}&147A\\
731:18&0,0,0,0,e^{12},e^{13},e^{35}+e^{14}+e^{26}&147A\\
731:19&0,0,0,0,e^{12},e^{34},e^{36}+e^{15}&137A\\
731:22a&0,0,0,0,e^{23}+e^{14},e^{24}+e^{13},e^{15}+e^{26}&137A\\
731:21&0,0,0,0,e^{12},e^{34},e^{25}+e^{13}+e^{46}&137B\\
731:24b&0,0,0,0,e^{24}-e^{13},-e^{12}+e^{34},e^{23}+e^{46}+e^{15}&137B\\[2pt]
\multicolumn{3}{c}{Type: $43$}\\*
73:3&0,0,0,0,e^{12},e^{13},e^{24}&37B\\
73:6a&0,0,0,0,e^{12},e^{14}+e^{23},e^{13}+e^{24}&37B\\
73:5&0,0,0,0,e^{12},e^{34},e^{13}+e^{24}&37D\\
73:7a&0,0,0,0,e^{23}+e^{14},e^{24}+e^{13},e^{12}+e^{34}&37D\\
\multicolumn{3}{c}{Type: $52$}\\*
72:2&0,0,0,0,0,e^{12},e^{34}&N_{3,2}\oplus N_{3,2}\oplus\R\\
72:4a&0,0,0,0,0,e^{24}+e^{13},e^{34}+e^{12}&N_{3,2}\oplus N_{3,2}\oplus\R\\
\end{longtable}
}
\end{footnotesize}

Comparing Table~\ref{tabella7NOnice} with Proposition~\ref{prop:two_inequivalent_bases}, we see that none of the central extensions of the $6$-dimensional Lie algebra $N_{6,1,4}$, namely $12457E$, $12457F$ and $12457G$, admits a nice basis. This is consistent with the following general fact:
\begin{proposition}
If $\g$ does not admit any nice basis, then no central extension of $\g$ admits a nice basis.
\end{proposition}
\begin{proof}
Follows from Lemma~\ref{lemma:adapted}.
\end{proof}

Using the list  of all  nice Lie algebras in dimension 7 (Table~\ref{TableNiceDim7}), one can check directly that all of them admit a derivation with nonzero trace: indeed, given a nice Lie algebra with diagram $\Delta$, any element of $\ker M_\Delta$ defines a diagonal derivation; it is then sufficient to verify that in each case $\ker M_\Delta$ is not contained in the hyperplane $x_1+\dotsc + x_n=0$. Since derivations with nonzero trace obstruct the existence of pseudoriemannian Einstein metrics with nonzero scalar curvature (see \cite[Theorem 4.1]{ContiRossi:EinsteinNilpotent}), we obtain an alternative, direct proof of  the following:
\begin{theorem}[{\cite[Theorem 5.4]{ContiRossi:EinsteinNilpotent}}]
Any left-invariant pseudoriemannian Einstein metric on a nice nilpotent Lie group of dimension 7 is Ricci-flat.
\end{theorem}

We conclude this section with some remarks on the relation between isomorphism and equivalence for nice nilpotent Lie algebras with fixed diagram. We have seen in Section~\ref{sec:algorithm} that two nice Lie algebras corresponding to distinct points of $\mathring{W}/\Aut(\Delta)$ are necessarily inequivalent. A natural question is whether they can be isomorphic --- in other words, whether a nilpotent Lie algebra can admit two inequivalent nice bases with the same diagram.

\begin{proposition}
\label{prop:finite7}
Given a nilpotent Lie algebra $\g$ of dimension up to $7$ and a nice diagram $\Delta$, any two nice bases on $\g$ with diagram $\Delta$ are equivalent.
\end{proposition}
\begin{proof}
Going through the classification, it suffices to check that for each diagram $\Delta$, the Lie algebras associated to different elements of $\mathring{W}$ are pairwise nonisomorphic, except when they are related by an automorphism of $\Delta$. In other words, we must check that whenever two nice Lie algebras with diagram $\Delta$ correspond to the same entry in Gong's list they are also related by $\Aut(\Delta)\ltimes D_n$.

This phenomenon appears exactly twice, for the diagrams \texttt{7431:13} and \texttt{741:6}. In the former case, we have two one-parameter families depending on a parameter, indicated by $A$ in Table~\ref{TableNiceDim7}. The action of the automorphism $(1 2)(5 6)$ on $\mathring{W}$ exchanges $A$ and $1/A$. By \cite{Gong}, nice Lie algebras in the family $1357S$  are pairwise isomorphic only when the invariant $\lambda=(1+A)^2/(1-A)^2$ attains the same value, and the nice Lie algebras in the family $1357QRS_1$ are pairwise isomorphic only when $\lambda+\frac1\lambda$ attains the same value. In both situations, this happens precisely for pairs $A,1/A$.

In the latter case, the family $147E$ has invariant $\frac{(1-\lambda+\lambda^2)^3}{\lambda^2(\lambda-1)^2}$. This means that $\lambda,\frac1\lambda, 1-\frac1{\lambda},  \frac{\lambda}{\lambda-1},1-\lambda,\frac1{1-\lambda}$ determine isomorphic Lie algebras. The group $\Aut(\Delta)$ has six elements and, for each $\lambda$, acts transitively on the corresponding elements of $\mathring{W}$.
\end{proof}

In higher dimensions, we have the following:
\begin{theorem}
Let $\mathring{W}\subset \mathring{V}_\Delta$ be a fundamental domain for the action of $e^{M_\Delta}(D_n)$ as in Proposition~\ref{prop:fundamental_domain}. Then at each $c\in \mathring{W}$
\begin{equation}
\label{eqn:transverse_orbits}
T_c{(\GL(n,\R)c)} \cap T_c\mathring{W}=\{0\}.
\end{equation}
In particular the map $\mathring{W}\to (\GL(n,\R) V_\Delta )/\GL(n,\R)$ has discrete fibers.
\end{theorem}
\begin{proof}
Fix $c=\sum c_IE_I\in \mathring{W}$; by hypothesis $d_n c\cap T_c\mathring{W}=\{0\}$. For any $A\in\gl(n,\R)$ such that $Ac\in V_\Delta$, write $A=A_{\text{diag}}+A_{\text{offdiag}}$, where $A_{\text{diag}}$ is diagonal  and $A_{\text{offdiag}}$ has zero on the diagonal. Then $A_{\diag}c\in V_\Delta$, so
\[A_{\text{offdiag}}c=Ac-A_{\diag}c \in V_\Delta.\]
By the nice condition, each $A_{\text{offdiag}}E_I$ lies in
\[\Span {e^{ij}\otimes e_k \mid \{\{i,j\},k\}\notin \mathcal{I}_\Delta}.\]
It follows that $A_{\text{offdiag}}c=0$. Thus, $Ac$ lies in $d_nc$, implying \eqref{eqn:transverse_orbits}.

The second claim follows from the fact that $\GL(n,\R)c\cap \mathring{W}$ is discrete in $\GL(n,\R)c$.
\end{proof}

\begin{corollary}
\label{cor:atmostcountable}
On a fixed nilpotent Lie algebra, the set of nice bases taken up to equivalence is at most countable.
\end{corollary}
\begin{proof}
Fix a nice basis $e_1,\dotsc, e_n$,  let $\Delta$ be the nice diagram, and let $c\in \mathring{V}_\Delta$ encode the structure constants. Let $\mathring{W}$ be a fundamental domain; up to equivalence, we may assume $c\in \mathring{W}$. Denote by $C$ the set of elements of $\mathring{W}$ that are in the same $\GL(n,\R)$-orbit as $c$. By the theorem, $C$ is a discrete subset of $\mathring{W}$, hence at most countable since $\mathring{W}$ is homeomorphic to some $(\R^*)^k$.

Any other nice basis with diagram $\Delta'$ isomorphic to $\Delta$ determines an element $c'\in V_{\Delta'}$, and through the induced isomorphism $V_\Delta\cong V_{\Delta'}$ an element $c''$ of $\mathring{V}_\Delta$.

Since $c$ and $c''$ define isomorphic Lie algebras, they are in the same $\GL(n,\R)$-orbit; up to equivalence, we may assume $c''\in \mathring{W}$, and so $c''\in C$. Thus, nice bases with diagram isomorphic to $\Delta$ are parametrized by $C$. It remains to observe that the set of isomorphism classes of nice diagrams with $n$ nodes is finite.
\end{proof}

\begin{remark}
We do not know any example of a nilpotent Lie algebra where the set of nice bases taken up to equivalence is infinite; however, we know that it can contain more than one element (see Proposition~\ref{prop:two_inequivalent_bases} or Theorem~\ref{thm:notnice7}).
\end{remark}

\begin{remark}
Continuous families of nice Lie algebras (taken up to equivalence) exist in any dimension $n\geq 7$, as one can see by taking the product of  \texttt{7421:14} with $\R^{n-7}$.

It follows from Corollary~\ref{cor:atmostcountable} that such continuous families correspond to continuous families in the category of Lie algebras (taken up to isomorphism). In particular, any continuous family contains Lie groups that do not admit a lattice, since connected, simply connected nilpotent Lie groups that admit a lattice are countably many (see \cite{Malcev}).
\end{remark}

\appendix
\section{Appendix}
\label{App:Tabelle}
We give the complete list of nice nilpotent Lie algebras up to dimension $7$; dimensions $8$ and $9$ are in the ancillary files. For dimension $1$ and $2$ the only nilpotent nice Lie algebras are isomorphic to $\R$ and $\R^2$. The columns in each table contain the name of the nice Lie algebra, the structure equations in the nice basis and the name of the Lie algebra in Gong's list (\cite{Gong}). When the Lie algebra is decomposable, the last column contains both the decomposition and the dimensions of the upper central series (which otherwise form part of the name in Gong's classification).

Recall that the name of each family contains a progressive number identifying the nice diagram; omitted integers correspond to nice diagrams that do not correspond to any nice Lie  algebra (see Remark~\ref{rmk:niceisnotnice}).

For exactly $20$ nice diagrams in dimension $9$, the Jacobi identity determines quadratic equations which survive the normalization; for $5$ of them the Jacobi identity has no admissible solution. For the remaining $15$, we have solved the quadratic equations and reduced the number of parameters; these entries are marked with a $(*)$ in the list. All the other families of nice Lie algebras correspond to affine subspace of $V_\Delta$, as produced directly by the algorithm.

\begin{table}[thp]
\centering\caption{3-dimensional nice nilpotent Lie algebras\label{table:3DimNiceGong}}
\begin{tabular}{>{\ttfamily}l L L}
\toprule
\textnormal{Name} & \g & \text{Gong \cite{Gong}} \\
\midrule
\multicolumn{3}{c}{Type: $21$}\\
31:1&0,0,e^{12}&13\ N_{3,2}\\[2pt] 
\multicolumn{3}{c}{Type: $3$}\\
3:1&0,0,0&3\ \R^3\\
\bottomrule
\end{tabular}
\end{table}

\begin{table}[thp]
\centering
\caption{4-dimensional nice nilpotent Lie algebras\label{table:4DimNiceGong}}
\begin{tabular}{>{\ttfamily}l L L}
\toprule
\textnormal{Name} & \g & \text{Gong \cite{Gong}} \\
\midrule
\multicolumn{3}{c}{Type: $211$}\\
421:1&0,0,e^{12},e^{13}&124\ N_{4,2}\\[2pt]
\multicolumn{3}{c}{Type: $31$}\\
41:1&0,0,0,e^{12}&24\ N_{3,2}\oplus\R\\[2pt]
\multicolumn{3}{c}{Type: $4$}\\
4:1&0,0,0,0&4\ \R^4\\
\bottomrule
\end{tabular}
\end{table}

\begin{table}[thp]
\centering
\caption{5-dimensional nice nilpotent Lie algebras\label{table:5DimNiceGong}}
\begin{tabular}{>{\ttfamily}l L L}
\toprule
\textnormal{Name} & \g & \text{Gong \cite{Gong}} \\
\midrule
\multicolumn{3}{c}{Type: $2111$}\\
5321:1&0,0,e^{12},e^{13},e^{14}&1235\ N_{5,2,1}\\
5321:2&0,0,e^{12},e^{13},e^{23}+e^{14}&1235\ N_{5,1}\\[2pt]
\multicolumn{3}{c}{Type: $212$}\\
532:1&0,0,e^{12},e^{13},e^{23}&235\ N_{5,2,3}\\[2pt]
\multicolumn{3}{c}{Type: $311$}\\
521:1&0,0,0,e^{12},e^{14}&235\ N_{4,2}\oplus\R\\
521:2&0,0,0,e^{12},e^{24}+e^{13}&135\ N_{5,2,2}\\[2pt]
\multicolumn{3}{c}{Type: $32$}\\
52:1&0,0,0,e^{12},e^{13}&25\ N_{5,3,2}\\[2pt]
\multicolumn{3}{c}{Type: $41$}\\
51:1&0,0,0,0,e^{12}&35\ N_{3,2}\oplus\R^2\\
51:2&0,0,0,0,e^{34}+e^{12}&15\ N_{5,3,1}\\[2pt]
\multicolumn{3}{c}{Type: $5$}\\
5:1&0,0,0,0,0&5\ \R^5\\
\bottomrule
\end{tabular}
\end{table}

\FloatBarrier
{\setlength{\tabcolsep}{2pt}
\begin{longtable}[htpc]{>{\ttfamily}l L L}
\caption{6 Dimensional nice Lie algebras\label{TableNiceDim6}}\\
\toprule
\textnormal{Name} & \g & \text{Gong \cite{Gong}} \\
\midrule
\endfirsthead
\multicolumn{3}{c}{\tablename\ \thetable\ -- \textit{Continued from previous page}} \\
\toprule
\textnormal{Name} & \g & \text{Gong \cite{Gong}} \\
\midrule
\endhead
\bottomrule\\[-7pt]
\multicolumn{3}{c}{\tablename\ \thetable\ -- \textit{Continued to next page}} \\
\endfoot
\bottomrule\\[-7pt]
\endlastfoot
\multicolumn{3}{c}{Type: $21111$}\\
64321:1&0,0,e^{12},e^{13},e^{14},e^{15}&12346\ N_{6,2,1}\\
64321:2&0,0,e^{12},e^{13},e^{14},e^{15}+e^{23}&12346\ N_{6,1,3}\\
64321:3&0,0,- e^{12},e^{13},e^{14},e^{34}+e^{25}&12346\ N_{6,2,2}\\
64321:4&0,0,e^{12},e^{13},e^{23}+e^{14},e^{24}+e^{15}&12346\ N_{6,1,1}\\
64321:5&0,0,- e^{12},e^{13},e^{14}+e^{23},e^{34}+e^{25}&12346\ N_{6,1,2}\\[2pt]
\multicolumn{3}{c}{Type: $2121$}\\
6431:1&0,0,e^{12},e^{13},e^{23},e^{14}&2346\ N_{6,2,7}\footnote{There is a misprint in \cite{Gong}, where the lower descending series is incorrectly written as $(6432)$.}\\
6431:2a&0,0,e^{12},e^{13},e^{23},e^{14}+e^{25}&1346\ N_{6,2,5a}\\
6431:2b&0,0,e^{12},- e^{13},e^{23},e^{14}+e^{25}&1346\ N_{6,2,5}\\
6431:3&0,0,e^{12},e^{13},e^{23},e^{24}+e^{15}&1346\ N_{6,2,5}\\[2pt]
\multicolumn{3}{c}{Type: $3111$}\\
6321:1&0,0,0,e^{12},e^{14},e^{15}&2346\ N_{5,2,1}\oplus\R\\ 
6321:2&0,0,0,e^{12},e^{14},e^{15}+e^{23}&1346\ N_{6,2,4}\\
6321:3&0,0,0,e^{12},e^{14},e^{24}+e^{15}&2346\ N_{5,1}\oplus\R\\
6321:4&0,0,0,- e^{12},e^{14}+e^{23},e^{34}+e^{15}&1246\ N_{6,2,3}\\[2pt]
\multicolumn{3}{c}{Type: $312$}\\
632:1&0,0,0,e^{12},e^{14},e^{24}&346\ N_{5,2,3}\oplus\R\\ 
632:2&0,0,0,e^{12},e^{14},e^{24}+e^{13}&246\ N_{6,2,10}\\
632:3a&0,0,0,e^{12},e^{14}+e^{23},e^{13}+e^{24}&246\ N_{6,2,9}\\
632:3b&0,0,0,e^{12},-e^{14}+e^{23},e^{24}+e^{13}&246\ N_{6,2,9a}\\[2pt]
\multicolumn{3}{c}{Type: $321$}\\
631:1&0,0,0,e^{12},e^{13},e^{14}&246\ N_{6,3,4}\\
631:2&0,0,0,e^{12},e^{13},e^{24}&246\ N_{6,3,3}\\
631:3&0,0,0,e^{12},e^{13},e^{23}+e^{14}&246\ N_{6,2,8}\\
631:4&0,0,0,e^{12},e^{13},e^{24}+e^{15}&136\ N_{6,2,6}\\
631:5a&0,0,0,e^{12},e^{13},e^{24}+e^{35}&136\ N_{6,3,1a}\\
631:5b&0,0,0,- e^{12},e^{13},e^{35}+e^{24}&136\ N_{6,3,1}\\
631:6&0,0,0,e^{12},e^{13},e^{25}+e^{34}&136\ N_{6,3,1}\\[2pt]
\multicolumn{3}{c}{Type: $33$}\\
63:1&0,0,0,e^{12},e^{13},e^{23}&36\ N_{6,3,6}\\[2pt]
\multicolumn{3}{c}{Type: $411$}\\
621:1&0,0,0,0,e^{12},e^{15}&346\ N_{4,2}\oplus\R^2\\
621:2&0,0,0,0,e^{12},e^{25}+e^{13}&246\ N_{5,2,2}\oplus\R\\
621:3&0,0,0,0,e^{12},e^{34}+e^{15}&146\ N_{6,3,2}\\[2pt]
\multicolumn{3}{c}{Type: $42$}\\
62:1&0,0,0,0,e^{12},e^{13}&36\ N_{5,3,2}\oplus \R\\
62:2&0,0,0,0,e^{12},e^{34}&26\ N_{3,2}\oplus N_{3,2}\\
62:3&0,0,0,0,e^{12},e^{24}+e^{13}&26\ N_{6,3,5}\\
62:4a&0,0,0,0,e^{13}+e^{24},e^{12}+e^{34}&26\ N_{3,2}\oplus N_{3,2}\\
62:4b&0,0,0,0,e^{24}-e^{13},e^{34}+e^{12}&26\ N_{6,4,4a}\\[2pt]
\multicolumn{3}{c}{Type: $51$}\\
61:1&0,0,0,0,0,e^{12}&46\ N_{3,2}\oplus\R^3\\
61:2&0,0,0,0,0,e^{34}+e^{12}&26\ N_{5,3,1}\oplus\R\\[2pt]
\multicolumn{3}{c}{Type: $6$}\\
6:1&0,0,0,0,0,0&6\ \R^6\\
\end{longtable}
}

\FloatBarrier
\begin{footnotesize}
{\setlength{\tabcolsep}{2pt}
\begin{longtable}[htpc]{>{\ttfamily}l  L L}
\caption{7~Dimensional nice Nilpotent Lie algebras\label{TableNiceDim7}}\\
\toprule
\textnormal{Name} & \g & \text{Gong \cite{Gong}} \\
\midrule
\endfirsthead
\multicolumn{3}{c}{\tablename\ \thetable\ -- \textit{Continued from previous page}} \\
\toprule
\textnormal{Name} & \g & \text{Gong \cite{Gong}} \\
\midrule
\endhead
\bottomrule\\[-7pt]
\multicolumn{3}{c}{\tablename\ \thetable\ -- \textit{Continued to next page}} \\
\endfoot
\bottomrule\\[-7pt]
\endlastfoot
\multicolumn{3}{c}{Type: $211111$}\\*
754321:1&0,0,e^{12},e^{13},e^{14},e^{15},e^{16}&123457A\\
754321:2&0,0,e^{12},e^{13},e^{14},e^{15},e^{23}+e^{16}&123457B\\
754321:3&0,0,e^{12},e^{13},e^{14},e^{23}+e^{15},e^{24}+e^{16}&123457D\\
754321:5&0,0,- e^{12},e^{13},e^{14},e^{15},e^{16}+e^{25}+e^{34}&123457C\\
754321:6&0,0,e^{12},e^{13},e^{23}+e^{14},e^{15}+e^{24},e^{16}+e^{34}&123457I\ \lambda=0\\
754321:7&0,0,e^{12},e^{13},e^{14}+e^{23},e^{15}+e^{24},e^{16}+e^{25}&123457I\ \lambda=1\\
\multirow{2}{*}{754321:9}&\multicolumn{1}{L}{0,0,{(1-\lambda)} e^{12},e^{13}, \lambda e^{14}+e^{23},}&\multirow{2}{*}{$123457I\ \lambda\neq 0,1$}\\*
&\multicolumn{1}{R}{e^{24}+e^{15},e^{25}+e^{34}+e^{16}}&\\[2pt]
\multicolumn{3}{c}{Type: $21112$}\\*
75432:1&0,0,- e^{12},e^{13},e^{14},e^{15},e^{25}+e^{34}&23457C\\
75432:2&0,0,- e^{12},e^{13},e^{14},e^{23}+e^{15},e^{34}+e^{25}&23457D\\
75432:3&0,0,- e^{12},e^{13},e^{14}+e^{23},e^{15}+e^{24},e^{34}+e^{25}&23457G\\[2pt]
\multicolumn{3}{c}{Type: $21211$}\\*
75421:1&0,0,e^{12},e^{13},e^{23},e^{14},e^{16}&23457A\\
75421:2&0,0,e^{12},e^{13},e^{23},e^{14},e^{16}+e^{25}&13457F\\
75421:3&0,0,- e^{12},e^{13},e^{23},e^{14},e^{26}+e^{34}&23457B\\
75421:4&0,0,e^{12},e^{13},e^{23},e^{15}+e^{24},e^{16}+e^{34}&12457H\\
75421:5a&0,0,- e^{12},e^{13},e^{23},e^{25}+e^{14},e^{26}+e^{34}&12457L_1\\
75421:5b&0,0,- e^{12},- e^{13},e^{23},e^{14}+e^{25},e^{26}+e^{34}&12457L\\
75421:6&0,0,- e^{12},e^{13},e^{23},e^{15}+e^{24},e^{14}+e^{35}+e^{26}&12457I\\[2pt]
\multicolumn{3}{c}{Type: $2122$}\\*
7542:1&0,0,e^{12},e^{13},e^{23},e^{14},e^{25}&2457L\\
7542:2&0,0,e^{12},e^{13},e^{23},e^{14},e^{15}+e^{24}&2457M\\
7542:3a&0,0,e^{12},e^{13},e^{23},e^{24}+e^{15},e^{14}+e^{25}&2457L\\
7542:3b&0,0,e^{12},- e^{13},e^{23},-e^{15}+e^{24},e^{14}+e^{25}&2457L_1\\[2pt] 
\multicolumn{3}{c}{Type: $31111$}\\*
74321:1&0,0,0,e^{12},e^{14},e^{15},e^{16}&23457\ N_{6,2,1}\oplus\R\\
74321:2&0,0,0,e^{12},e^{14},e^{15},e^{16}+e^{23}&13457A\\
74321:3&0,0,0,e^{12},e^{14},e^{15},e^{24}+e^{16}&23457\ N_{6,1,3}\oplus\R\\
74321:4&0,0,0,- e^{12},e^{14},e^{15},e^{26}+e^{45}&23457\ N_{6,2,2}\oplus\R\\
74321:5&0,0,0,- e^{12},e^{14},e^{15}+e^{23},e^{16}+e^{34}&12457A\\
74321:6&0,0,0,- e^{12},e^{14},e^{15}+e^{23},e^{45}+e^{26}&12457C\\
74321:7&0,0,0,- e^{12},e^{14},e^{15},e^{26}+e^{45}+e^{13}&13457C\\
74321:8&0,0,0,e^{12},e^{14},e^{24}+e^{15},e^{16}+e^{25}&23457\ N_{6,1,1}\oplus\R\\
74321:9&0,0,0,- e^{12},e^{14},e^{15}+e^{24},e^{26}+e^{45}&23457\ N_{6,1,2}\oplus\R\\
74321:10&0,0,0,- e^{12},e^{14},e^{15}+e^{23},e^{45}+e^{13}+e^{26}&12457D\\
74321:11&0,0,0,- e^{12},e^{14},e^{24}+e^{15},e^{13}+e^{45}+e^{26}&13457E\\
74321:12&0,0,0,- e^{12},e^{23}+e^{14},e^{34}+e^{15},e^{35}+e^{16}&12357A\\
74321:15&0,0,0,- e^{12},e^{14}+e^{23},e^{15}+e^{34},e^{35}+e^{16}+e^{24}&12357C\\[2pt]
\multicolumn{3}{c}{Type: $3121$}\\*
7431:1&0,0,0,e^{12},e^{14},e^{24},e^{15}&3457 \ N_{6,2,7}\oplus\R\\
7431:2&0,0,0,e^{12},e^{14},e^{24}+e^{13},e^{15}&2457G\\
7431:3&0,0,0,e^{12},e^{14},e^{24},e^{23}+e^{15}&2457H\\
7431:4&0,0,0,e^{12},e^{14},e^{13}+e^{24},e^{15}+e^{23}&2457K\\
7431:5&0,0,0,e^{12},e^{13}+e^{24},e^{14},e^{34}+e^{25}&2357C\\
7431:6a&0,0,0,- e^{12},e^{14}+e^{23},e^{24}+e^{13},e^{15}+e^{34}&2357D\\
7431:6b&0,0,0,- e^{12},e^{23}-e^{14},e^{24}+e^{13},e^{15}+e^{34}&2357D_1\\
7431:7a&0,0,0,e^{12},e^{14},e^{24},e^{26}+e^{15}&2457\ N_{6,2,5a}\oplus\R\\
7431:7b&0,0,0,e^{12},- e^{14},e^{24},e^{26}+e^{15}&2457\ N_{6,2,5}\oplus\R\\
7431:8&0,0,0,e^{12},e^{14},e^{24},e^{25}+e^{16}&2457\ N_{6,2,5}\oplus\R\\
7431:9&0,0,0,e^{12},e^{14},e^{24}+e^{13},e^{16}+e^{25}&1357O\\
7431:10a&0,0,0,e^{12},e^{14}+e^{23},e^{13}+e^{24},e^{15}+e^{26}&1357Q\\
7431:10b&0,0,0,e^{12},-e^{14}+e^{23},e^{13}+e^{24},e^{26}+e^{15}&1357QRS_1\ \lambda=1\\
7431:11a&0,0,0,e^{12},e^{24}+e^{13},e^{14}+e^{23},e^{15}+e^{26}&1357Q_1\\
7431:11b&0,0,0,e^{12},e^{24}-e^{13},e^{23}+e^{14},e^{15}+e^{26}&1357QRS_1\ \lambda=1\\
7431:12a&0,0,0,e^{12},e^{14},e^{13}+e^{24},e^{26}+e^{15}+e^{34}&1357P\\
7431:12b&0,0,0,e^{12},- e^{14},e^{13}+e^{24},e^{34}+e^{26}+e^{15}&1357P_1\\
\multirow{2}{*}{7431:13a$\left\{\!\!\begin{array}{c}
\!\\
\!\\
\!\\
\!\\
\end{array}\!\!\right.\!\!$}&\multicolumn{1}{L}{0,0,0, (A-1) e^{12},e^{14}+e^{23},\hspace{58pt} A\!<\!0}&\left\{\!\begin{array}{c}
						1357QRS_1\\
						\lambda\!=\!A\!<\!0\\
						\end{array}\right.\\*
&\multicolumn{1}{R}{e^{24}+A e^{13},e^{26}+e^{34}+e^{15}\quad A\!\neq\!1,A\!>\!0}&\left\{\!\begin{array}{c}
						1357S\\
						\lambda\!=\!\frac{(1+A)^2}{(A-1)^2}\!>\!1\\
						\end{array}\right.\\
\multirow{2}{*}{7431:13b$\left\{\!\!\begin{array}{c}
\!\\
\!\\
\!\\
\!\\
\!\\
\end{array}\!\!\right.\!\!$}&\multicolumn{1}{L}{0,0,0, (A-1) e^{12},e^{23}-e^{14},\hspace{35pt} A\!\neq\!1,A\!>\!0}&\left\{\!\begin{array}{c}
						1357QRS_1\\
						\lambda\!=\!A\!>\!0, \lambda\! \neq\! 1\\
						\end{array}\right.\\*  
&\multicolumn{1}{R}{e^{24}+A e^{13},e^{15}+e^{26}+e^{34}\quad A\!\neq\!-1,A\!<\!0}&\left\{\!\begin{array}{c}
						1357S\\
						0\!<\!\lambda\!=\!\frac{(1+A)^2}{(A-1)^2}\!<\!1\\
						\end{array}\right.\\*
&0,0,0, -2 e^{12},e^{23}-e^{14},e^{24}- e^{13},e^{15}+e^{26}+e^{34}&1357R\\[2pt]
\multicolumn{3}{c}{Type: $3211$}\\*
7421:1&0,0,0,e^{12},e^{13},e^{14},e^{16}&2457A\\
7421:2&0,0,0,e^{12},e^{13},e^{24},e^{26}&2457B\\
7421:3&0,0,0,e^{12},e^{13},e^{14},e^{16}+e^{23}&2457C\\
7421:4&0,0,0,e^{12},e^{13},e^{14},e^{16}+e^{24}&2457F\\
7421:5&0,0,0,e^{12},e^{13},e^{24},e^{26}+e^{14}&2457I\\
7421:6&0,0,0,e^{12},e^{13},e^{14},e^{16}+e^{35}&1357E\\
7421:7&0,0,0,e^{12},e^{13},e^{24},e^{26}+e^{35}&1357I\\
7421:8&0,0,0,e^{12},e^{13},e^{24},e^{15}+e^{26}&1357G\\
7421:9&0,0,0,- e^{12},e^{13},e^{23}+e^{14},e^{34}+e^{16}&2357B\\
7421:10&0,0,0,e^{12},e^{13},e^{14}+e^{23},e^{25}+e^{16}&1357M\ \lambda=1\\
7421:11a&0,0,0,e^{12},e^{13},e^{14},e^{16}+e^{24}+e^{35}&1357F\\
7421:11b&0,0,0,- e^{12},e^{13},e^{14},e^{16}+e^{35}+e^{24}&1357F_1\\
7421:12&0,0,0,e^{12},e^{13},e^{14},e^{34}+e^{16}+e^{25}&1357D\\
7421:13&0,0,0,e^{12},e^{13},e^{24},e^{35}+e^{14}+e^{26}&1357J\\
7421:14&0,0,0, (\lambda-1) e^{12}, \lambda e^{13},e^{14}+e^{23},e^{25}+e^{34}+e^{16}&1357M\ \lambda\neq0,1\\[2pt]
\multicolumn{3}{c}{Type: $322$}\\*
742:1&0,0,0,e^{12},e^{13},e^{14},e^{24}&357B\\
742:2&0,0,0,e^{12},e^{13},e^{24},e^{23}+e^{14}&357C\\
742:3&0,0,0,e^{12},e^{13},e^{14},e^{15}&247A\\
742:4&0,0,0,e^{12},e^{13},e^{24},e^{15}&247B\\
742:5&0,0,0,e^{12},e^{13},e^{24},e^{35}&247F\\
742:6&0,0,0,e^{12},e^{13},e^{14},e^{23}+e^{15}&247L\\
742:7&0,0,0,e^{12},e^{13},e^{24},e^{15}+e^{23}&247M\\
742:8&0,0,0,e^{12},e^{13},e^{14},e^{15}+e^{24}&247C\\
742:9a&0,0,0,e^{12},e^{13},e^{14},e^{24}+e^{35}&247E_1\\
742:9b&0,0,0,- e^{12},e^{13},e^{14},e^{35}+e^{24}&247E\\
742:10&0,0,0,e^{12},e^{13},e^{14},e^{34}+e^{25}&247D\\
742:11&0,0,0,e^{12},e^{13},e^{24},e^{35}+e^{14}&247G\\
742:12&0,0,0,e^{12},e^{13},e^{24},e^{25}+e^{34}&247I\\
742:13&0,0,0,e^{12},e^{13},e^{23}+e^{14},e^{24}+e^{15}&247O\\
742:14a&0,0,0,e^{12},e^{13},e^{23}+e^{14},e^{35}+e^{24}&247R_1\\
742:14b&0,0,0,- e^{12},e^{13},e^{14}+e^{23},e^{24}+e^{35}&247R\\
742:15&0,0,0,e^{12},e^{13},e^{14}+e^{23},e^{25}+e^{34}&247Q\\
742:16&0,0,0,e^{12},e^{13},e^{15}+e^{24},e^{35}+e^{14}&247H\\
742:17&0,0,0,e^{12},e^{13},e^{24}+e^{15},e^{34}+e^{25}&247K\\
742:18a&0,0,0,e^{12},e^{13},e^{25}+e^{34},e^{24}+e^{35}&247F\\
742:18b&0,0,0,- e^{12},e^{13},-e^{25}+e^{34},e^{24}+e^{35}&247F_1\\[2pt]
\multicolumn{3}{c}{Type: $331$}\\*
741:1&0,0,0,e^{12},e^{13},e^{23},e^{14}&357A\\
741:2&0,0,0,e^{12},e^{13},e^{23},e^{24}+e^{15}&247N\\
741:3a&0,0,0,e^{12},e^{13},e^{23},e^{24}+e^{35}&247P_1\\
741:3b&0,0,0,- e^{12},e^{13},e^{23},e^{24}+e^{35}&247P\\
741:4&0,0,0,e^{12},e^{13},e^{23},e^{25}+e^{34}&247P\\
741:5&0,0,0,e^{12},e^{13},e^{23},e^{15}+e^{36}+e^{24}&147E_1\ \lambda=2\\ 
741:6&0,0,0, (\lambda-1)e^{12},  \lambda e^{13},e^{23},e^{34}+e^{25}+e^{16}&147E\ \lambda\neq0,1\\[2pt]
\multicolumn{3}{c}{Type: $4111$}\\*
7321:1&0,0,0,0,e^{12},e^{15},e^{16}&3457\ N_{5,2,1}\oplus\R^2\\
7321:2&0,0,0,0,e^{12},e^{15},e^{16}+e^{34}&1457A\\
7321:3&0,0,0,0,e^{12},e^{15},e^{16}+e^{23}&2457\ N_{6,2,4}\oplus\R\\
7321:4&0,0,0,0,e^{12},e^{15},e^{25}+e^{16}&3457\ N_{5,1}\oplus\R^2 \\
7321:5&0,0,0,0,e^{12},e^{15},e^{16}+e^{25}+e^{34}&1457B\\
7321:6&0,0,0,0,- e^{12},e^{15}+e^{23},e^{16}+e^{35}&2357\ N_{6,2,3}\oplus\R\\
7321:7&0,0,0,0,e^{12},e^{25}+e^{13},e^{26}+e^{35}+e^{14}&1357A\\[2pt]
\multicolumn{3}{c}{Type: $412$}\\*
732:1&0,0,0,0,e^{12},e^{15},e^{25}&457\ N_{5,2,3}\oplus\R^2\\
732:2&0,0,0,0,e^{12},e^{15},e^{13}+e^{25}&357\ N_{6,2,10}\oplus\R\\
732:3&0,0,0,0,e^{12},e^{15},e^{25}+e^{34}&257K\\
732:4a&0,0,0,0,e^{12},e^{15}+e^{23},e^{13}+e^{25}&357\ N_{6,2,9}\oplus\R\\ 
732:4b&0,0,0,0,e^{12},-e^{15}+e^{23},e^{13}+e^{25}&357\ N_{6,2,9a}\oplus\R\\
732:5&0,0,0,0,e^{12},e^{25}+e^{13},e^{15}+e^{34}&257L\\
732:6&0,0,0,0,e^{12},e^{15}+e^{23},e^{25}+e^{14}&257J\\[2pt]
\multicolumn{3}{c}{Type: $421$}\\*
731:1&0,0,0,0,e^{12},e^{13},e^{15}&357\ N_{6,3,4}\oplus\R\\
731:2&0,0,0,0,e^{12},e^{13},e^{25}&357\ N_{6,3,3}\oplus\R\\
731:3&0,0,0,0,e^{12},e^{34},e^{15}&257\ N_{4,2}\oplus N_{3,2}\\
731:4&0,0,0,0,e^{12},e^{13},e^{23}+e^{15}&357\ N_{6,2,8}\oplus\R\\
731:5&0,0,0,0,e^{12},e^{13},e^{34}+e^{15}&257F\\
731:6&0,0,0,0,e^{12},e^{13},e^{25}+e^{34}&257E\\
731:7&0,0,0,0,e^{12},e^{24}+e^{13},e^{15}&257B\\
731:8&0,0,0,0,e^{12},e^{34},e^{25}+e^{13}&257H\\
731:9&0,0,0,0,e^{12},e^{13},e^{25}+e^{14}&257C\\
731:10&0,0,0,0,e^{12},e^{13},e^{15}+e^{24}&257A\\
731:11&0,0,0,0,e^{12},e^{24}+e^{13},e^{34}+e^{15}&257G\\
731:12&0,0,0,0,e^{12},e^{23}+e^{14},e^{13}+e^{25}&257D\\
731:13&0,0,0,0,e^{12},e^{13},e^{25}+e^{16}&247\ N_{6,2,6}\oplus\R\\
731:14a&0,0,0,0,e^{12},e^{13},e^{36}+e^{25}&247\ N_{6,3,1a}\oplus\R\\
731:14b&0,0,0,0,- e^{12},e^{13},e^{36}+e^{25}&247\ N_{6,3,1}\oplus\R\\
731:15&0,0,0,0,e^{12},e^{13},e^{26}+e^{35}&247\ N_{6,3,1}\oplus\R\\
731:16a&0,0,0,0,e^{12},e^{13},e^{14}+e^{36}+e^{25}&147A_1\\
731:16b&0,0,0,0,- e^{12},e^{13},e^{36}+e^{25}+e^{14}&147A\\
731:17&0,0,0,0,e^{12},e^{13},e^{24}+e^{36}+e^{15}&147B\\
731:18&0,0,0,0,e^{12},e^{13},e^{35}+e^{14}+e^{26}&147A\\
731:19&0,0,0,0,e^{12},e^{34},e^{36}+e^{15}&137A\\
731:20&0,0,0,0,- e^{12},e^{14}+e^{23},e^{35}+e^{16}&137C\\
731:21&0,0,0,0,e^{12},e^{34},e^{25}+e^{13}+e^{46}&137B\\
731:22a&0,0,0,0,e^{23}+e^{14},e^{24}+e^{13},e^{15}+e^{26}&137A\\
731:22b&0,0,0,0,e^{23}-e^{14},e^{24}+e^{13},e^{26}+e^{15}&137A_1\\
731:23&0,0,0,0,e^{12},e^{13}+e^{24},e^{14}+e^{26}+e^{35}&137D\\
731:24a&0,0,0,0,e^{24}+e^{13},-e^{12}+e^{34},e^{46}+e^{15}+e^{23}&137B_1\\
731:24b&0,0,0,0,e^{24}-e^{13},-e^{12}+e^{34},e^{23}+e^{46}+e^{15}&137B\\[2pt]
\multicolumn{3}{c}{Type: $43$}\\*
73:1&0,0,0,0,e^{12},e^{13},e^{23}&47\ N_{6,3,6}\oplus\R\\
73:2&0,0,0,0,e^{12},e^{13},e^{14}&37A\\
73:3&0,0,0,0,e^{12},e^{13},e^{24}&37B\\
73:4&0,0,0,0,e^{12},e^{13},e^{23}+e^{14}&37C\\
73:5&0,0,0,0,e^{12},e^{34},e^{13}+e^{24}&37D\\
73:6a&0,0,0,0,e^{12},e^{14}+e^{23},e^{13}+e^{24}&37B\\
73:6b&0,0,0,0,e^{12},-e^{14}+e^{23},e^{24}+e^{13}&37B_1\\
73:7a&0,0,0,0,e^{23}+e^{14},e^{24}+e^{13},e^{12}+e^{34}&37D\\
73:7b&0,0,0,0,e^{14}+e^{23},e^{24}-e^{13},e^{12}+e^{34}&37D_1\\
\multicolumn{3}{c}{Type: $511$}\\*
721:1&0,0,0,0,0,e^{12},e^{16}&457\ N_{4,2}\oplus\R^3\\
721:2&0,0,0,0,0,e^{12},e^{26}+e^{13}&357\ N_{5,2,2}\oplus\R^2\\
721:3&0,0,0,0,0,e^{12},e^{16}+e^{34}&257\ N_{6,3,2}\oplus\R\\
721:4&0,0,0,0,0,e^{12},e^{45}+e^{13}+e^{26}&157\\[2pt]
\multicolumn{3}{c}{Type: $52$}\\*
72:1&0,0,0,0,0,e^{12},e^{13}&47\ N_{5,3,2}\oplus\R^2\\
72:2&0,0,0,0,0,e^{12},e^{34}&37\ N_{3,2}\oplus N_{3,2}\oplus\R\\
72:3&0,0,0,0,0,e^{12},e^{13}+e^{24}&37\ N_{6,3,5}\oplus\R\\
72:4a&0,0,0,0,0,e^{24}+e^{13},e^{34}+e^{12}&37\ N_{3,2}\oplus N_{3,2}\oplus\R\\
72:4b&0,0,0,0,0,e^{24}-e^{13},e^{34}+e^{12}&37\ N_{6,4,4a}\oplus\R\\
72:5&0,0,0,0,0,e^{12},e^{13}+e^{45}&27A\\
72:6&0,0,0,0,0,e^{13}+e^{24},e^{35}+e^{12}&27B\\[2pt]
\multicolumn{3}{c}{Type: $61$}\\*
71:1&0,0,0,0,0,0,e^{12}&57\ N_{3,2}\oplus\R^4\\
71:2&0,0,0,0,0,0,e^{12}+e^{34}&37\ N_{5,3,1}\oplus\R^2\\
71:3&0,0,0,0,0,0,e^{12}+e^{34}+e^{56}&17\\[2pt]
\multicolumn{3}{c}{Type: $7$}\\*
7:1&0,0,0,0,0,0,0&7\ \R^7\\
\end{longtable}
}
\end{footnotesize}

\bibliographystyle{plain}
\bibliography{construction.bbl}

\def\cprime{$'$}
\begin{thebibliography}{10}

\bibitem{Andrada:ComplexProduct}
A.~Andrada.
\newblock Complex product structures on 6-dimensional nilpotent {L}ie algebras.
\newblock {\em Forum Math.}, 20(2):285--315, 2008.

\bibitem{AngellaRossi:DComplexCohomology}
D.~Angella and F.~A. Rossi.
\newblock Cohomology of {$\mathbf{D}$}-complex manifolds.
\newblock {\em Differential Geom. Appl.}, 30(5):530--547, 2012.

\bibitem{BazzoniFernandezMunoz:NonFormalCoSymplectic}
G.~Bazzoni, M.~Fern\'andez, and V.~Mu\~noz.
\newblock Non-formal co-symplectic manifolds.
\newblock {\em Trans. Amer. Math. Soc.}, 367(6):4459--4481, 2015.

\bibitem{BCFG:BoundaryHitchinHypoFlow}
F.~Belgun, V.~Cort\'es, M.~Freibert, and O.~Goertsches.
\newblock On the boundary behavior of left-invariant {H}itchin and hypo flows.
\newblock {\em J. Lond. Math. Soc. (2)}, 92(1):41--62, 2015.

\bibitem{CavalcantiGualtieri:GeneralizedNilmanifolds}
G.~R. Cavalcanti and M.~Gualtieri.
\newblock Generalized complex structures on nilmanifolds.
\newblock {\em J. Symplectic Geom.}, 2(3):393--410, 2004.

\bibitem{ContiFernandez:calibrated}
D.~Conti and M.~Fern\'andez.
\newblock Nilmanifolds with a calibrated {$G_2$}-structure.
\newblock {\em Differential Geom. Appl.}, 29(4):493--506, 2011.

\bibitem{ContiMadsen:Harmonic}
D.~Conti and T.~B. Madsen.
\newblock Harmonic structures and intrinsic torsion.
\newblock {\em Transform. Groups}, 20(3):699--723, 2015.

\bibitem{ContiRossi:EinsteinNice}
D.~Conti and F.~A. Rossi.
\newblock Indefinite {E}instein metrics on nice {L}ie groups.
\newblock arXiv:1805.08491.

\bibitem{ContiRossi:ricci}
D.~Conti and F.~A. Rossi.
\newblock The {R}icci tensor of almost parahermitian manifolds.
\newblock {\em Ann. Global Anal. Geom.}, 53(4):467--501, JUN 2018.

\bibitem{ContiRossi:EinsteinNilpotent}
D.~Conti and F.~A. Rossi.
\newblock Einstein nilpotent {L}ie groups.
\newblock {\em J. Pure Appl. Algebra}, 223(3):976--997, 2019.

\bibitem{CorderoFernandezUgarte}
L.~A. Cordero, M.~Fern\'andez, and L.~Ugarte.
\newblock Pseudo-{K}\"ahler metrics on six-dimensional nilpotent {L}ie
  algebras.
\newblock {\em J. Geom. Phys.}, 50(1-4):115--137, 2004.

\bibitem{DottiFino:AbelianHypercomplex}
I.~G. Dotti and A.~Fino.
\newblock Abelian hypercomplex 8-dimensional nilmanifolds.
\newblock {\em Ann. Global Anal. Geom.}, 18(1):47--59, 2000.

\bibitem{Fernandez:Calibrated}
M.~Fern{\'a}ndez.
\newblock An example of a compact calibrated manifold associated with the
  exceptional {L}ie group {$G\sb 2$}.
\newblock {\em J. Differential Geom.}, 26(2):367--370, 1987.

\bibitem{FernandezCulma}
E.~A. Fern\'andez-Culma.
\newblock Classification of 7-dimensional {E}instein nilradicals.
\newblock {\em Transform. Groups}, 17(3):639--656, 2012.

\bibitem{Fernandez-Culma2015:SolitonOn6DimNilmanifolds}
E.~A. Fern{\'a}ndez-Culma.
\newblock Soliton almost {K}{\"a}hler structures on 6-dimensional nilmanifolds
  for the symplectic curvature flow.
\newblock {\em {J}. {G}eom. {A}nal.}, 25(4):2736--2758, Oct 2015.

\bibitem{FinoLujan:TorsionFreeG22}
A.~Fino and I.~Luj\'an.
\newblock Torsion-free {$G^*_{2(2)}$}-structures with full holonomy on
  nilmanifolds.
\newblock {\em Adv. Geom.}, 15(3):381--392, 2015.

\bibitem{FinoPartonSalamon}
A.~Fino, M.~Parton, and S.~Salamon.
\newblock Families of strong {KT} structures in six dimensions.
\newblock {\em Comment. Math. Helv.}, 79(2):317--340, 2004.

\bibitem{Gong}
M.-P. Gong.
\newblock {\em Classification of nilpotent {L}ie algebras of dimension 7 (over
  algebraically closed fields and {R})}.
\newblock ProQuest LLC, Ann Arbor, MI, 1998.
\newblock Thesis (Ph.D.)--University of Waterloo (Canada).

\bibitem{Heber:noncompact}
J.~Heber.
\newblock Noncompact homogeneous {E}instein spaces.
\newblock {\em Invent. Math.}, 133(2):279--352, 1998.

\bibitem{KadiogluPayne:Computational}
H.~Kadioglu and T.~L. Payne.
\newblock Computational methods for nilsoliton metric {L}ie algebras {I}.
\newblock {\em J. Symbolic Comput.}, 50:350--373, 2013.

\bibitem{Lauret:Einstein_solvmanifolds}
J.~Lauret.
\newblock Einstein solvmanifolds are standard.
\newblock {\em Ann. of Math. (2)}, 172(3):1859--1877, 2010.

\bibitem{LauretWill:Einstein}
J.~Lauret and C.~Will.
\newblock Einstein solvmanifolds: existence and non-existence questions.
\newblock {\em Math. Ann.}, 350(1):199--225, 2011.

\bibitem{LauretWill:diagonalization}
J.~Lauret and C.~Will.
\newblock On the diagonalization of the {R}icci flow on {L}ie groups.
\newblock {\em Proc. Amer. Math. Soc.}, 141(10):3651--3663, 2013.

\bibitem{Magnin}
L.~Magnin.
\newblock Sur les alg\`ebres de {L}ie nilpotentes de dimension {$\leq 7$}.
\newblock {\em J. Geom. Phys.}, 3(1):119--144, 1986.

\bibitem{Malcev}
A.~I. Malcev.
\newblock On a class of homogeneous spaces.
\newblock {\em Amer. Math. Soc. Translation}, 1951(39):33, 1951.

\bibitem{Nikolayevsky}
Y.~Nikolayevsky.
\newblock Einstein solvmanifolds and the pre-{E}instein derivation.
\newblock {\em Trans. Amer. Math. Soc.}, 363(8):3935--3958, 2011.

\bibitem{Nomizu:cohomology}
K.~Nomizu.
\newblock On the cohomology of compact homogeneous spaces of nilpotent {L}ie
  groups.
\newblock {\em Ann. of Math. (2)}, 59:531--538, 1954.

\bibitem{Payne:ExistenceOfSolitonMetrics}
T.~L. Payne.
\newblock The existence of soliton metrics for nilpotent {L}ie groups.
\newblock {\em Geom. Dedicata}, 145:71--88, 2010.

\bibitem{Payne:Applications}
T.~L. Payne.
\newblock Applications of index sets and {N}ikolayevsky derivations to positive
  rank nilpotent {L}ie algebras.
\newblock {\em J. Lie Theory}, 24(1):1--27, 2014.

\bibitem{Payne:Methods}
T.~L. Payne.
\newblock Methods for parametrizing varieties of {L}ie algebras.
\newblock {\em J. Algebra}, 445:1--34, 2016.

\bibitem{Rollenske:GeometryNilmanifold}
S.~Rollenske.
\newblock Geometry of nilmanifolds with left-invariant complex structure and
  deformations in the large.
\newblock {\em Proc. Lond. Math. Soc. (3)}, 99(2):425--460, 2009.

\bibitem{RossiTomassini}
F.~A. Rossi and A.~Tomassini.
\newblock On strong {K}\"ahler and astheno-{K}\"ahler metrics on nilmanifolds.
\newblock {\em Adv. Geom.}, 12(3):431--446, 2012.

\bibitem{Salamon:ComplexStructures}
S.~Salamon.
\newblock Complex structures on nilpotent {L}ie algebras.
\newblock {\em J. Pure Appl. Algebra}, 157:311--333, 2001.

\bibitem{Hengesbach}
F.~Schulte-Hengesbach.
\newblock Half-flat structures on products of three-dimensional lie groups.
\newblock {\em J. Geom. Phys}, 60(11):1726 -- 1740, 2010.

\end{thebibliography}

\small\noindent Dipartimento di Matematica e Applicazioni, Universit\`a di Milano Bicocca, via Cozzi 55, 20125 Milano, Italy.\\
\texttt{diego.conti@unimib.it}\\
\texttt{federico.rossi@unimib.it}

\end{document}